\documentclass{article}
\usepackage[all]{xy}
\usepackage[usenames]{xcolor}
\usepackage{theorem}
\usepackage{amsmath}
\usepackage{amsfonts}
\usepackage{amscd}
\usepackage{amssymb}
\input xypic

\newcommand{\mlabel}[1]{{\color{white} \marginpar{#1}}}

\newtheorem{defn}{Definition}[section]
\newtheorem{thm}[defn]{Theorem}
\newtheorem{prop}[defn]{Proposition}
\newtheorem{lemma}[defn]{Lemma}
{\theorembodyfont{\rmfamily}\newtheorem{eg}[defn]{Example}}
\newtheorem{claim}[defn]{Claim}
\newtheorem{cor}[defn]{Corollary}
\newtheorem{conj}[defn]{Conjecture}

\newenvironment{proof}{\paragraph{Proof.}}{\hfill$\Box$\null}

\newcommand{\lm}{\ensuremath{\longrightarrow}}
\newcommand{\Rhom}{\mathbf{R}\strut\kern-.2em\operatorname{Hom}\nolimits}

\newcommand{\vs}{\vspace{2mm}}

\newcommand{\h}{h}
\renewcommand{\L}{\Lambda}

\newcommand{\xlm}{\ensuremath{\xrightarrow}}

\newcommand{\eps}{\varepsilon}

\newcommand{\Hom}{\operatorname{Hom}\nolimits}

\DeclareMathOperator{\shom}{\ensuremath{\mathcal{H}\mathit{om}}}

\DeclareMathOperator{\send}{\ensuremath{\mathcal{E}\!\mathit{nd}}}

\newcommand{\Ext}{\operatorname{Ext}\nolimits}
\newcommand{\End}{\operatorname{End}\nolimits}

\DeclareMathOperator{\proj}{\mbox{Proj}\,}
\DeclareMathOperator{\mo}{\mathsf{mod}}

\DeclareMathOperator{\Pic}{\mbox{Pic}\,}

\DeclareMathOperator{\Div}{\mbox{Div}\,}
\DeclareMathOperator{\rk}{\mbox{rank}\,}
\DeclareMathOperator{\coker}{\mbox{coker}\,}

\DeclareMathOperator{\tr}{\mbox{tr}\,}

\DeclareMathOperator{\Coh}{\text{\sf Coh}}

\DeclareMathOperator{\mmm}{\mathfrak{m}}
\DeclareMathOperator{\s}{\sigma}

\DeclareMathOperator{\PP}{\mathbb{P}}

\DeclareMathOperator{\Z}{\mathbb{Z}}
\DeclareMathOperator{\Q}{\mathbb{Q}}

\DeclareMathOperator{\X}{\mathbb{X}}

\DeclareMathOperator{\calm}{\mathcal{M}}
\DeclareMathOperator{\calo}{\mathcal{O}}

\DeclareMathOperator{\caln}{\mathcal{N}}
\DeclareMathOperator{\call}{\mathcal{L}}
\DeclareMathOperator{\calp}{\mathcal{P}}

\DeclareMathOperator{\calt}{\mathcal{T}}
\DeclareMathOperator{\Xt}{\tilde{X}}
\DeclareMathOperator{\Dt}{\tilde{D}}
\DeclareMathOperator{\oa}{\mathcal{O}_{A}}
\DeclareMathOperator{\ox}{\mathcal{O}_{X}}
\DeclareMathOperator{\oxt}{\mathcal{O}_{\Xt}}

\DeclareMathOperator{\oc}{\mathcal{O}_{C}}
\DeclareMathOperator{\bfx}{\mathbf{x}}
\DeclareMathOperator{\bfy}{\mathbf{y}}

\begin{document}

\begin{center}
\LARGE \textbf{2-hereditary algebras and almost Fano weighted surfaces}
\end{center}

\begin{center}
  DANIEL CHAN
\footnote{This project was supported by the Australian Research Council, Discovery Project Grant DP130100513.}
\end{center}

\begin{center}
  {\em University of New South Wales}
\end{center}

\begin{center}
e-mail address:{\em danielc@unsw.edu.au}
\end{center}

\begin{abstract}
Tilting bundles $\calt$ on a weighted projective line $\mathbb{X}$ have been intensively studied by representation theorists since they give rise to a derived equivalence between $\mathbb{X}$ and the finite dimensional algebra $\End \calt$. A classical result states that if $\End \calt$ is hereditary, then $\mathbb{X}$ is Fano and conversely, for every Fano weighted projective line, there exists a tilting bundle $\calt$ with $\End \calt$ hereditary. In this paper, we examine the question of when a weighted projective surface has a tilting bundle whose endomorphism ring is 2-hereditary in the sense of Herschend-Iyama-Oppermann. It is natural to conjecture that they are the almost Fano weighted surfaces, weighted only on rational curves, and we give evidence to support this. 
\end{abstract}

Throughout, we work over an algebraically closed base field $k$ of characteristic zero.

\section{Introduction}  \label{sintro}  \mlabel{sintro}

The representation theory of finite dimensional algebras is surprisingly subtle, as the complexity of module categories can vary vastly and chaotically with respect to changes in generators and relations. Tilting theory provides a way of explaining why some algebras have a nice representation theory: if an algebra is the endomorphism ring of a tilting bundle $\calt$ on a nice (weighted) projective variety $\X$, then there is a derived equivalence $D^b_{fg}(\End \calt) \simeq D^b_c(\X)$ which allows us to ``import'' the good representation theory of $\X$ to $\End \calt$. The classic examples are Ringel's canonical algebras $\Lambda$ \cite{R}, which all arise as endomorphism algebras of tilting bundles on a weighted projective line $\X$ as defined by Geigle and Lenzing \cite{GL} and this gives a nice geometric explanation of the nice representation theory of $\Lambda$. Furthermore, $\Lambda$ is derived equivalent to an hereditary algebra if and only if $\X$ is Fano. The purpose of this paper is to examine this picture in the case of weighted projective surfaces. 

More precisely, the notion of being hereditary has been generalised by Herschend-Iyama-Opperman to the notion of being $n$-hereditary (see \cite{HIO}) which is a little stronger than having global dimension $n$. Weighted projective spaces and varieties have also been studied under a variety of different contexts (\cite{HIMO}, \cite{IL}, \cite{LO}). They can be viewed as log varieties of the form $(X,\Delta = \sum_i (1 -\frac{1}{p_i})D_i)$ (with some smoothness assumptions) or certain associated orders dubbed Geigle-Lenzing orders by Iyama-Lerner \cite{IL}. The log variety viewpoint allows us to consider geometric concepts such as (almost) Fano (see Section~\ref{slog}) and the log minimal model program, whilst the order approach immediately gives us a category of coherent sheaves so we can talk about tilting bundles for weighted projective varieties as well as Serre duality. We consider the question: ``Which weighted projective surfaces $\mathbb{X}$ have tilting bundles $\calt$ such that $\End \calt$ is 2-hereditary?''. We call such a tilting bundle {\em $2$-hereditary}. The related question of which projective surfaces have tilting bundles, has also been studied for quite some time. Hille and Perling \cite{HP} show that all rational surfaces have tilting bundles, but it is still an open conjecture if these are the only ones. Our study suggests the following

\begin{conj}   \label{cmain}  \mlabel{cmain} 
A weighted projective surface $\X$ has a 2-hereditary tilting bundle if and only if the weighted divisors $D_i$ are rational and $\X$ is almost Fano. 
\end{conj}

Note that the conditions of being Fano and almost Fano coincide for curves, so this is an analogue of the dimension one result above. In \cite{HIMO}, it is shown that every weighted projective plane which is weighted on $\leq 3$ lines (and so in particular is Fano), has a 2-hereditary tilting bundle. It is also easy to construct a 2-hereditary tilting bundle on the Hirzebruch surface $\mathbb{F}_2$ which is not Fano. 

In support of the conjecture we prove the following  
\begin{thm}  \label{tnec} \mlabel{tnec}
 [Corollary~\ref{cneccriterion}, Theorem~\ref{tdemazure}] Let $(X,\Delta)$ be a weighted projective surface with a 2-hereditary tilting bundle. Then $-(K_X + \Delta)$ is nef and $(K_X + \Delta)^2 \geq 0$. In particular, if $\Delta = 0$, then either $X$ is almost Fano, or it is a blowup of $\PP^2$ at 9 points in almost general position (defined in Definition~\ref{dKXonC}). 
\end{thm}

We have the following partial converse which shows that in the non-weighted case, the conjecture cannot be too far from the truth. 

\begin{thm}  \label{tDP}  \mlabel{tDP}
Let $X$ be a Fano surface of degree $K_X^2 \geq 3$. Then $X$ has a 2-hereditary tilting bundle (explicitly described in Theorem~\ref{tdpbundles}) which is a direct sum of line bundles. 
\end{thm}

Perhaps, at the end of the day, it is the examples of $2$-hereditary algebras which are the most important, and the main purpose of the conjecture and results towards it, are to guide this search. One nice feature of constructing $n$-hereditary algebras via tilting bundles is that they are automatically $n$-representation infinite and, in the Fano case, immediately tame by Proposition~\ref{ptame}.

There are several standard ways to look for a tilting bundle $\calt$, and we use them to find $n$-hereditary ones. The easiest is to look for direct sums of line bundles. Checking the generation condition is often easy as Lerner-Oppermann \cite{LO} give a set of generators for the derived category of a weighted projective variety, a result we reprove in Section~\ref{sgen}. To check the other conditions then amounts to a cohomology calculation which is easy for line bundles (see Proposition~\ref{plinebundles}). In the dimension one case, there are Fano weighted projective lines which have no 1-hereditary tilting bundle of this form. These tilting bundles can be obtained by either mutating other tilting bundles or using group actions. We also investigate these methods to produce examples of 2-hereditary tilting bundles and their associated 2-hereditary algebras. Our results in this direction are more scattered, and we give by way of illustration, the following.

\begin{thm}  \label{texamples}  \mlabel{texamples}
 The following Fano weighted projective surfaces have 2-hereditary tilting bundles.
\begin{enumerate}
 \item (Proposition~\ref{p11ram}) $\PP^1 \times \PP^1$ weighted on a smooth $(1,1)$-divisor.
 \item (Corollary~\ref{cconicram}) $\PP^2$ weighted on a conic with weight 2. 
 \item (Theorem~\ref{t2222}) $\PP^2$ weighted on 4 lines in general position with all weights equal to 2.
\end{enumerate}
\end{thm}
It is interesting to note that in cases ii) and iii) above, there do not exist 2-hereditary tilting bundles which are direct sums of line bundles. 

\vs
\noindent
\textbf{Notation:} Throughout, we work over an algebraically closed base field $k$ of characteristic zero. The symbol $X$ will always denote a smooth projective variety and $A$, a finite sheaf of algebras on $X$ (in fact, it will always be an order on $X$ as defined in Section~\ref{sGLorders}). We will by default work with left $A$-modules and let $A-\mo$ denote the category of coherent $A$-modules. We let $D^b_c(A)$ denote that bounded derived category of $A-\mo$. 

\section{Geigle-Lenzing orders}  \label{sGLorders} \mlabel{sGLorders}

In this section, we recall the approach in \cite{IL} of studying weighted projective varieties via GL-orders and collect some basic facts regarding this viewpoint. We end with a simple necessary condition for a weighted projective variety to have an $n$-hereditary tilting bundle. 

A {\em (smooth) weighted projective variety} $\mathbb{X}$ consists of the data of a smooth projective variety $X$, a finite set of smooth divisors $D_1,\ldots, D_r$ and integers $p_1,\ldots, p_r \geq 2$ called the {\em weights of the $D_i$}, such that $\cup D_i$ is simple normal crossing, that is, complete locally at any point of $X$, $\cup D_i$ is isomorphic to a union of coordinate hyperplanes. Its {\em dimension} is the dimension of the underlying variety $X$. When the dimension is 2, we will speak of {\em weighted projective surfaces}. We will only consider smooth weighted projective varieties, so we will omit the adjective smooth in future. For reasons that will become clear in Section~\ref{slog}, we will use the notation
$$\mathbb{X} = (X,\sum_i (1 - \tfrac{1}{p_i})D_i) $$
to denote our weighted projective variety. We will often write $\Delta$ for $\sum_i (1 - \tfrac{1}{p_i})D_i)$ in keeping with standard notation in the theory of log varieties. 

The easiest way to associate a category of coherent sheaves to a weighted projective variety is to introduce orders. Let $k(X)$ denote the function field of a smooth projective variety $X$. 
An {\em order} on $X$ is a torsion free coherent sheaf of algebras $A$ on $X$ such that $A \otimes_X k(X)$ is a central simple $k(X)$-algebra. In this paper, we will only deal with the case where $A \otimes_X k(X)$ is a full matrix algebra in $k(X)$ so $A$ embeds in $M_m(k(X))$ for some $m$. 

Let $D$ be a smooth effective divisor on $X$ and $p \geq 2$ be an integer. As in \cite[Section~2]{IL}, we consider the following subalgebra of the matrix algebra $M_p(\ox)$
$$ T_p(D) := 
\begin{pmatrix}
 \ox & \ox(-D) & \ldots & \ox(-D) & \ox(-D) \\
 \ox & \ox & \ldots & \ox(-D) & \ox(-D) \\
 \vdots & \vdots & \ldots &  & \vdots \\
 \ox & \ox & \ldots & \ox & \ox(-D) \\
 \ox & \ox & \ldots & \ox & \ox \\
\end{pmatrix}
$$ 
The {\em standard Geigle-Lenzing (GL)-order} associated to the weighted projective variety $\mathbb{X}$ above is
$$A = T_{p_1}(D_1) \otimes_X \ldots \otimes_XT_{p_r}(D_r) .$$
The {\em category of coherent sheaves on $\mathbb{X}$}, {\sf coh\,}$\mathbb{X}$ can now be defined to be the category $A-\mo$ of coherent $A$-modules. Since we will only be interested in $A$ as far as its category of coherent sheaves is concerned, we will call any algebra Morita equivalent to $A$ a {\em GL-order}. Our $A$-modules will usually be viewed as $(A,\ox)$-bimodules, where the $\ox$-action is central. We collect some basic facts about $A$.

\begin{prop}  \label{pGLorders} \mlabel{pGLorders}
 Let $A$ be a GL-order associated to the weighted projective variety $\mathbb{X}$ above.
\begin{enumerate}
 \item $A$ is an order which embeds in $M_{p}(k(X))$ for some $p\in \mathbb{N}$. 
\item The stalk $A_x$ of $A$ at any closed point $x \in X$ is a ring with global dimension $\dim X$. 
\end{enumerate}
\end{prop}
\begin{proof} It suffices to assume that $A$ is standard. 
 Part i) is easy while part ii) is proved when $X = \mathbb{P}^n$ in \cite[Theorem~2.2c)]{IL}. The proof holds in our case too verbatim, for the global dimension of $A_x$ can be computed complete locally at $x$ and our simple normal crossing assumption on the weighted divisors corresponds precisely to the general position hypothesis in \cite[Assumption~2.1]{IL}. 
\end{proof}\vs

As in \cite{IL}, we will need to consider ``line bundles'' over a GL-order $A$ and their Ext groups. Recall that an $A$-module $\calp$ is {\em locally projective} if the stalk $\calp_x$ at any closed point $x \in X$ is projective as an $A_x$-module. Now $A$ embeds in the full matrix algebra $M_p(k(X))$. Hence the rank of $\calp$ as a sheaf on $X$, denoted here by $\rk_X \calp$, is a multiple of $p$ and we define the {\em $A$-rank} of $\calp$ to be the integer
$$ \rk_A \calp = \tfrac{1}{p}\rk_X \calp .$$
This is a Morita invariant so only depends on the weighted projective variety, not the choice of GL-order representing it. When $\rk_A \calp = 1$, we will often refer to $\calp$ as a {\em line bundle} on the weighted projective variety or $A$. 

To generate line bundles, we start with some auto-equivalences of $A-\mo$. Naturally, given a line bundle $\caln$ on $X$, $-\otimes_X \caln$ induces an auto-equivalence, but the weighted divisors $D_i$ give some extra ones as follows. It is easiest to describe these when $A$ is a standard GL-order, so we assume this for now. Then we have an invertible ideal 

$$J_i := 
\begin{pmatrix}
 \ox(-D_i) & \ox(-D_i) & \ldots & \ox(-D_i) & \ox(-D_i) \\
 \ox & \ox(-D_i) & \ldots & \ox(-D_i) & \ox(-D_i) \\
 \vdots & \vdots & \ldots &  & \vdots \\
 \ox & \ox & \ldots & \ox(-D_i) & \ox(-D_i) \\
 \ox & \ox & \ldots & \ox & \ox(-D_i) \\
\end{pmatrix}
\triangleleft T_{p_i}(D_i)
$$
which satisfies $J_i^{p_i} = T_{p_i}(D_i) \otimes_X \ox(-D_i)$. Hence 
$$I_i := T_{p_1}(D_1) \otimes_X \ldots \otimes_X T_{p_{i-1}}(D_{i-1}) \otimes_X J_i \otimes_X T_{p_{i+1}}(D_{i+1}) \otimes_X\ldots \otimes_X T_{p_r}(D_r)$$ 
is an invertible ideal of $A$ and for any divisor $D \in \text{Div} X$, we obtain the following auto-equivalences of $A-\mo$
$$\calm \mapsto \calm(D + \sum\tfrac{l_i}{p_i}D_i) = \prod_i I_i^{-l_i} \otimes_A \calm \otimes_X \ox(D).$$
This naturally extends existing notation. 

One disadvantage with viewing weighted projective varieties via orders, is that there is no distinguished candidate for the structure sheaf. Instead, we simply declare
$$\oa : = \ox^{p_1} \otimes_X \ldots \otimes_X \ox^{p_r} $$
which is a coherent $A$-module since each $\ox^{p_i}$ is an $M_{p_i}(\ox)$-module and hence a $T_{p_i}(D_i)$-module. It is furthermore locally projective, being a direct summand of $A$. Hence we obtain line bundles $\oa(D + \sum\frac{l_i}{p_i}D_i)$  on $A$. In fact, up to isomorphism, these are all the line bundles by \cite[proof of Proposition~3.2]{C}. The following proposition shows how to compute the Ext groups between line bundles. Its elementary proof is immediate and can for example be found in \cite[Proposition~3.1]{C} and \cite[Lemmas~2.7 and 2.8]{IL} (the slightly more general notation here allows their proof in the weighted projective space case to apply here too). 

\begin{prop}  \label{plinebundles} \mlabel{plinebundles}
 \begin{enumerate}
  \item Let $\calp,\calp'$ be two locally projective $A$-modules. Then the sheaf of $A$-module homomorphisms $\shom_A(\calp,\calp')$ is a vector bundle on $X$ whose $i$-th cohomology is $\Ext^i_A(\calp,\calp')$. 
 \item If $0 \leq l_i < p_i$ and $D$ is a divisor on $X$, then $\shom_A(\oa, \oa(D + \sum\frac{l_i}{p_i}D_i)) = \ox(D)$. 
 \end{enumerate}
\end{prop}

One feature of orders $A$ on $X$ is that they have a {\em trace map} $\tr:A \lm \ox$ which is $\ox$-linear. For GL-orders, $A$ embeds in the matrix algebra $M_p(k(X))$ and $\tr$ is just the restriction of the usual trace map. When $A$ is a standard GL-order, it is easy to see the image lies in $\ox$. In general, this follows from the theory of orders. We also obtain a trace pairing $A \times A \lm \ox: (a,b) \mapsto \tr(ab)$.  
The module category $A-\mo$ also has Serre duality. Indeed, the following is well-known.

\begin{prop}  \label{pSerredual}  \mlabel{pSerredual}
 Let $A$ be a $GL$-order associated to the smooth projective variety $X$ weighted on $D_1,\ldots, D_r$ with weights $p_1,\ldots, p_r$. 
\begin{enumerate}
 \item $A-\mo$ has a Serre functor given by $\nu:=\omega_A \otimes_A (-) [\dim X]$ where $\omega_A$ is the {\em canonical $A$-bimodule} defined by 
$$ \omega_A = \shom_X(A,\omega_X) .$$
 \item This bimodule can be computed in terms of the weights as 
$$\omega_A = A(K_X + \sum_i (1 - \tfrac{1}{p_i})D_i)$$
where $K_X$ is a canonical divisor on $X$. In particular, 
$$\nu = (-)(K_X + \sum_i (1 - \tfrac{1}{p_i})D_i)[\dim X]$$
\end{enumerate}
\end{prop}
\begin{proof}
The proof of part i) can be found in \cite[Proposition~6.5]{LO}. Part ii) is a computation using the trace pairing.
\end{proof}\vs

\section{The Cox ring}  \label{scox}  \mlabel{scox}

The Cox ring is a very useful tool to aid in the calculation of endomorphism algebras of tilting bundles and more generally, Ext spaces. In this section, we give an elementary version for the GL-orders we are interested in. 

Let $X$ be a smooth projective variety where $\Pic X$ is a free abelian group. Then we may lift $\Pic X$ to a subgroup $\Pic^{\uparrow} X$ of $\Div X$, the group of divisors on $X$. We will view the Cox ring of $X$ as the $\Pic X$-graded algebra 
$$ R_X :=\bigoplus_{\call \in \text{Pic} X} H^0(\call) .$$
To describe the multiplication, we identify $R_X$ with $\bigoplus_{D \in \text{Pic}^{\uparrow} X} H^0(\ox(D))$. Now the $\ox(D)$ are all subsheaves of the constant sheaf $k(X)$, so multiplication in $k(X)$ induces an associative multiplication on $R_X$. We will identify the Picard group  of $X$ with its divisor class group and, given a divisor $D \subset X$, let $[D]$ denote the associated divisor class. 

Let $\mathbb{X}$ be the weighted projective variety obtained by weighting $X$ at the divisors $D_1,\ldots, D_r$ with weights $p_1,\ldots, p_r$ and $A$ be the corresponding standard GL-order. Let $\mathbb{L}$ be the group generated by $\Pic X$ and the $[\tfrac{1}{p_i}D_i], i = 1,2, \ldots, r$ subject to the relations $p_i[\tfrac{1}{p_i}D_i] = [D_i]$. We call $\mathbb{L}$ the {\em divisor class group} of $\mathbb{X}$. When there is no chance of confusion, we will abuse notation by dropping brackets, and write for example  $\tfrac{1}{p_i}D_i \in \mathbb{L}$. It will be convenient to introduce the following notation
$$[0,\Delta] : = \left\{  \sum_{i=1}^r \tfrac{l_i}{p_i}D_i \mid 0 \leq l_i  \leq 1 - \tfrac{1}{p_i}, l_i \in \Z\right\}$$
To define the Cox ring of $\mathbb{X}$, we first need to introduce the {\em sheaf Cox ring} of $A$ which is the following $\mathbb{L}$-graded sheaf of algebras on $X$. 
$$\mathcal{R}_A := \bigoplus_{D \in \text{Pic}^{\uparrow} X} \bigoplus_{D' \in [0,\Delta]} A(D + D').$$
Abusing notation as above, we will often just write $\mathcal{R}_A = \bigoplus_{D \in \mathbb{L}} A(D)$. To define the multiplication, we need to choose an isomorphism $\s_i: \ox(D_i) \simeq \ox(D_i^{\uparrow})$ where $D_i^{\uparrow} \in \Pic^{\uparrow} X$. We view $\s_i$ as an element of $k(X)$. Given $D_1,D_2 \in \Pic^{\uparrow} X$ and $D'_1 = \sum \tfrac{l_{1i}}{p_i}D_i ,D'_2= \sum \tfrac{l_{2i}}{p_i} \in [0,\Delta]$, the multiplication is given by the composite
$$A(D_1 + D'_1) \otimes_X A(D_2 + D'_2) \xrightarrow{\text{nat}} A(D_1 + D_2 + D'_1 + D'_2) \xrightarrow{\s} A(D + D') $$
where $D \in \Pic^{\uparrow}X, D' \in [0,\Delta]$ are chosen so that $D_1 + D_2 + D'_1 + D'_2 \sim D + D'$ and $\s = \prod_{i \in I} \s_i$ where
$$I = \{i \mid l_{i1} + l_{i2} \geq p_i \}.$$
We omit the elementary verification that multiplication is associative. Note that changing the isomorphisms $\s_i$ gives an isomorphic graded algebra. 

We define the {\em Cox ring} of $A$ to be
$$ R_A := \Hom_A(A, \mathcal{R}_A) = H^0(\mathcal{R}_A) = \bigoplus_{D \in \mathbb{L}} H^0(A(D)) $$
and note that its degree zero component is $H^0(A)$. The algebra structure on $\mathcal{R}_A$ induces an algebra structure on $R_A$. Let $e \in H^0(A)$ be the primitive diagonal idempotent corresponding to the direct summand $\calo_A$ of $A$. We then define the {\em Cox ring} of $\mathbb{X}$ or {\em reduced Cox ring} of $A$ to be
$$ R_{\mathbb{X}} := e R_A e = \bigoplus_{D \in \mathbb{L}} \Hom_A(\calo_A,\calo_A(D)).$$
The $\mathbb{L}$-graded left $A$-module $\bigoplus_{D \in \mathbb{L}} \calo_A(D)$ is naturally a right $R_{\mathbb{X}}$-module. Note  that $R_X$ is just the subalgebra of $R_{\mathbb{X}}$ consisting of the components in $\Pic X \subseteq \mathbb{L}$. Also, if we picked a different primitive diagonal idempotent to $e$, we would have obtained an isomorphic algebra. 

\section{$n$-hereditary tilting bundles}  \label{snhereditary}  \mlabel{snhereditary}

We first recall a generalisation of hereditary algebras due to Herschend-Iyama-Opperman. Let $\Lambda$ be a finite dimensional $k$-algebra of global dimension $n$. It has a Serre functor 
$$\nu = D \circ \Rhom_{\L}(-,\L): D^b(\Lambda) \lm D^b(\Lambda)$$
where $D$ denotes the $k$-linear dual and $D^b(\L)$ denotes the bounded derived category of $\L$-modules. Let $\nu_n = \nu \circ [-n]$. 

\begin{defn}  \label{dnhered}  \mlabel{dnhered}
(\cite[Proposition~3.2, Definition~2.7]{HIO}) We say $\L$ is {\em $n$-hereditary} if for all $i \geq 0$, $\nu_n^{-i} \L$ has cohomology only in degrees divisible by $n$, that is, $H^j(\nu_n^{-i}) = 0$ if $n\nmid j$. If furthermore, $\nu_n^{-i}(\L)$ is always a module, then we say that $\L$ is {\em $n$-representation infinite}. In this case, we say $\Lambda$ is {\em tame} if the preprojective algebra 
$$\Pi(\Lambda)  = \bigoplus_i\Hom(\Lambda, \nu_n^{-i}\Lambda)$$
is a finitely generated module over a commutative noetherian ring. 
\end{defn}

These algebras have particularly nice properties and it is natural to use tilting theory to try to construct lots of interesting examples. We consider hence a weighted projective variety $\mathbb{X}$ with associated GL-order $A$. Using our order point of view, the definition of a tilting bundle is as follows

\begin{defn}  \label{dtilting}  \mlabel{dtilting}
 A {\em tilting bundle} on $\mathbb{X}$ or $A$ is a locally projective coherent $A$-module $\calt$ such that $\Ext_A^i(\calt,\calt) = 0$ for all $i \geq 0$ and $\calt$ generates $D^b_c(A)$.  
\end{defn}

Tilting bundles give rise to $n$-hereditary algebras under the following conditions.

\begin{prop} \label{pnheredbundle}  \mlabel{pnheredbundle}
 Let $\calt$ be a tilting bundle on a GL-order $A$ on $X$ and $n=\dim X$. Then $\L = \End_A \calt$ is $n$-hereditary if and only if it satisfies the following two conditions
\begin{enumerate}
 \item $\Ext^j_A(\calt, \omega_A^{-1} \otimes_A \calt) = 0$ for all $j > 0$ and,
 \item $\Ext^j_A(\calt, \omega_A^{-i} \otimes_A \calt) = 0$ for all $i \geq 2$ and $0 < j < n$.  
\end{enumerate}
In this case, $\L$ is actually $n$-representation infinite and $\Ext^n_A(\calt, \omega_A^{-i} \otimes_A \calt) = 0$ for all $i \geq 2$. Furthermore, $\Hom_A(\calt, \omega_A^{i} \otimes_A \calt) = 0$ for all $i > 0$.

Finally, $\Lambda$ is tame if and only if the orbit algebra $\bigoplus_i \Hom_A(\calt, \omega_A^{-i} \otimes_A \calt)$ is a finitely generated module over a commutative noetherian ring. 
\end{prop}
\begin{proof}
Since $\calt$ is a tilting bundle, $\Rhom_A(\calt,-): D^b_{c}(A) \lm D^b_{fg}(\L)$ is a derived equivalence and in particular, commutes with the Serre functor and shifts. Hence 
\begin{equation}
\nu_n^{-i} \L = \Rhom_A(\calt, \omega_A^{\otimes -i} \otimes_A \calt).
\label{epnheredbundle}
\end{equation}
 Putting $i = 1$, condition i) corresponds to the fact that $\nu_n^{-1}(\L) = \Rhom_{\L}(D\L,\L)[n]$ is a module so $\L$ has global dimension $n$ as observed in \cite[Observation~2.5]{HIMO}. Conditions i) and ii) now correspond precisely to the fact that $H^j(\nu_n^{-i}\L) = 0$ if $n \nmid j$. 

If $\L$ is not $n$-representation infinite, then there is some indecomposable projective $\L$-module $P$ and $i >0$ such that $\nu_n^{-i} P$ is not a module. Then by \cite[Lemma~3.5]{HIO}, we must have $\nu_n^{-i+1}P$ is an injective module. But then $\nu_n^{-i} P = \Rhom_{\L}(D \nu_{n}^{-i+1}P, \L)[n]$ is a non-zero complex living in degree $-n$. This contradicts equation~(\ref{epnheredbundle}) so $\L$ must also be $n$-representation infinite and $\Ext^n_A(\calt, \omega_A^{-i} \otimes_A \calt) = 0$ for all $i \geq 1$. 

Serre duality for $A-\mo$ now shows that for $i \geq 1$ we have 
$$  0  = \Hom_A(\omega_A^{-i} \otimes_A \calt, \omega_A \otimes_A \calt) = \Hom_A(\calt, \omega_A^{i+1} \otimes_A \calt).$$

The final assertion characterising when $\Lambda$ is tame is now clear from the above discussion. 
\end{proof}\vs

Recall that if a weighted projective curve has a tilting bundle $\calt$ with hereditary endomorphism ring $\L$, then $\L$ is of infinite representation type. The above proposition generalises this fact. We are led to the following 

\begin{defn}  \label{dnheredbundle}  \mlabel{hnheredbundle}
 Let $X$ be a smooth projective variety of dimension $n$. A tilting bundle $\calt$ on a GL-order $A$ on $X$ (or weighted projective variety) is {\em $n$-hereditary} if $\Ext^j_A(\calt, \omega_A^{-i} \otimes_A \calt) = 0$ for all $i,j \geq 0$.
\end{defn}

The next result gives an easy necessary criterion for a weighted projective variety to have an $n$-hereditary tilting bundle. 
\begin{cor}  \label{cneccriterion}  \mlabel{cneccriterion}
Let $\mathbb{X} = (X,\sum_i (1 - \frac{1}{p_i})D_i)$ be a weighted projective variety of dimension $n$. Suppose that $\mathbb{X}$ has an $n$-hereditary tilting bundle. Let $\Delta = \sum_i (1 - \frac{1}{p_i})D_i$. Then for all $j > 0$ and non-negative common multiples $r$ of all the $p_i$, we have 
$$H^j(\ox(-r(K_X + \Delta))) = 0.$$ 
Furthermore, $H^0(\ox(r(K_X + \Delta))) = 0$ for any positive common multiple $r$ of all the $p_i$.
\end{cor}
\begin{proof}
Let $\calt$ be an $n$-hereditary tilting bundle. Then the sheaf of algebras $E = \send_A \calt$ is an order on $X$, and the trace map $\tr: E \lm \ox$ splits since we are assuming our base field $k$ has characteristic zero. Furthermore, Proposition~\ref{pSerredual} and the local-global Ext-spectral sequence ensure 
$$\Ext^j_A(\calt, \omega_A^{-r} \otimes_A \calt) = \Ext^j_A(\calt, \ox(-r(K_X + \Delta))  \otimes_X \calt) = H^j (E \otimes_X \ox(r(K_X + \Delta)) ).$$
Proposition~\ref{pnheredbundle} now gives the desired vanishing of cohomology. 
\end{proof}

\section{Almost Fano weighted projective varieties}  \label{slog}  \mlabel{slog}

Our formula for the canonical bimodule of a GL-order and Corollary~\ref{cneccriterion} suggest an intimate connection between weighted projective varieties and log varieties as studied by algebraic geometers. In this section we recall some basic aspects of the log minimal model program and apply it to give evidence for Conjecture~\ref{cmain}. For the benefit of representation theorists, we include some basic definitions and refer them to \cite{KM} for further details. 

Let $X$ be a smooth projective variety of dimension $n>0$ and $D$ be a $\Q$-divisor on $X$. When given an expression such as $H^i(\ox(rD))$, we will say the integer $r$ is {\em appropriately divisible} if $rD$ is an integral divisor. We say that $D$ is {\em nef} if for every irreducible curve $C \subset X$ we have $D.C \geq 0$. We say $D$ is {\em big} if for appropriately divisible $r$, the function $\dim_k H^0(\ox(rD))$ grows as a polynomial of degree $n$. If $D$ is nef, this is equivalent to the fact that the intersection product $D^n >0$ \cite[Proposition~2.61]{KM}. Any ample divisor is big and nef, and in the case of surfaces, the Nakai-Moishezon criterion \cite[Theorem~V.1.10]{Har} asserts that $D$ is ample if and only if $D$ is big and $D.C >0$ for all irreducible curves $C$. 

Let $\left(X,\Delta = \sum_i(1 - \frac{1}{p_i}D_i)\right)$ be a weighted projective variety and $A$ be its associated standard GL-order. Prompted by Proposition~\ref{pSerredual}, we let
$$K_A = K_X + \Delta.$$

\begin{defn} \label{dfano}  \mlabel{dfano}
We say that $A$ or $(X,\Delta)$ is {\em Fano} if $-K_A$ is ample, and {\em almost Fano} if $-K_A$ is big and nef. 
\end{defn}

We say $A$ has {\em negative Kodaira dimension} if the log variety $(X,\Delta)$ has negative Kodaira dimension, that is, $H^0(\ox(rK_A)) = 0$ for any positive appropriately divisible integer $r$. Otherwise, it has {\em non-negative Kodaira dimension} . Corollary~\ref{cneccriterion} shows that any weighted projective variety with an $n$-hereditary tilting bundle has negative Kodaira dimension. 

Our first piece of evidence for our main Conjecture~\ref{cmain} is the following result. 

\begin{prop}  \label{pnec}  \mlabel{pnec}
Let $(X,\Delta)$ be a weighted projective surface. Suppose that 
\begin{itemize}
 \item $(X,\Delta)$ has negative Kodaira dimension and,
 \item $H^1(\ox(-r(K_X + \Delta))) = 0$ for all appropriately divisible non-negative integers $r$.
\end{itemize}
In particular, by Corollary~\ref{cneccriterion}, these two hypotheses hold whenever $(X,\Delta)$ has a 2-hereditary tilting bundle. 

Then $-(K_X + \Delta)$ is nef and $(K_X + \Delta)^2 \geq 0$. 
\end{prop}
\begin{proof}
If $(K_X + \Delta)^2 <0$, then by Riemann-Roch, we would have $\chi(\ox(-r(K_X + \Delta))) \lm -\infty$ for appropriately divisible $r$ tending towards $\infty$. This contradicts $H^1(\ox(-r(K_X + \Delta))) = 0$ so we may assume $(K_X + \Delta)^2 \geq 0$. 

Suppose that $-(K_X + \Delta)$ is not nef, so there is an irreducible curve $C$ such that $(K_X + \Delta). C >0$. From the exact sequence of sheaves
$$ 0 \lm \ox(-r(K_X + \Delta) -C) \lm \ox(-r(K_X + \Delta)) \lm \oc(-r(K_X + \Delta)) \lm 0 $$
we obtain an exact sequence in cohomology
$$ H^1(\ox(-r(K_X + \Delta))) \lm H^1(\oc(-r(K_X + \Delta))) \lm H^2(\ox(-r(K_X + \Delta) -C)).$$
Now $\deg (-r(K_X + \Delta)|_C) \lm - \infty$ so for sufficiently large and divisible $r$ we have $H^1(\oc(-r(K_X + \Delta))) \neq 0$. It suffices to show that the right hand term is zero to arrive at a contradiction. Serre duality gives
$$ \dim H^2(\ox(-r(K_X + \Delta) -C))  = \dim H^0(\ox(K_X + r(K_X + \Delta) +C)).$$
Now note that our simple normal crossing assumption on $\Delta$ ensures that $(X,\Delta)$ is klt \cite[Corollary~2.31(3)]{KM}, so we may apply the results of the log minimal model program to it. The easy dichotomy theorem \cite[Theorem~1-3-9 and p. 126]{Mat} shows that there is an open subset $U$ of $X$ such that for any point $p \in U$, there is an irreducible curve $H$ passing through $p$ such that $(K_X + \Delta).H <0$. Furthermore, any two such curves (possibly passing through different points) are numerically proportional to each other so $H^2\geq 0$. If we choose $r$ sufficiently large and divisible so that $(K_X + r(K_X + \Delta) +C).H <0$, then we must have $H^0(\ox(K_X + r(K_X + \Delta) +C))=0$ as desired. 
\end{proof}
\vs

We have some more restrictions on when 2-hereditary tilting bundles exist. 

\begin{prop} \label{pcentres}  \mlabel{pcentres}
 Let $(X,\Delta)$ be a weighted projective surface of negative Kodaira dimension. Suppose that $(X,\Delta)$ has a tilting bundle. Then $X$ is a rational surface. 
\end{prop}
\begin{proof}
 Since $\Delta$ is an effective divisor, $X$ also has negative Kodaira dimension so is birationally ruled over a smooth curve, say $C$. Let $A$ be the standard GL-order associated to $(X,\Delta)$. Recall that $\oa$ is a locally projective $A$-module so we have a pair of exact adjoint functors
$$ \oa \otimes_X -: \Coh (X) \lm A-\mo , \quad \shom_A(\oa, -): A-\mo \lm \Coh(X).$$
They induce maps on Grothendieck groups, say $\pi^*: K_0(X) \lm K_0(A), \pi_*: K_0(A) \lm K_0(X)$. For any vector bundle $\mathcal{V}$ on $X$, there is a natural isomorphism $\shom_A(\oa,\oa \otimes_X \mathcal{V}) \simeq \mathcal{V}$ so $\pi^*$ is a split injection. Since $A$ has a tilting bundle, $K_0(A)$ is finitely generated so the same is true of $K_0(X)$. But $\text{Pic}\, C$ embeds in $K_0(C)$ by \cite[Exercise~II.6.11]{Har} which in turn embeds in $K_0(X)$ \cite[end of Example~15.1.1]{Ful} so $C$ must be rational as must be $X$. 
\end{proof}\vs

It seems desirable to have a classification of almost Fano weighted projective surfaces, as well as those which satisfy more generally, the hypotheses of Proposition~\ref{pnec}. Almost Fano surfaces were studied classically by Demazure \cite{Dem} (see also \cite{HW}). Most are obtained by blowing up $\PP^2$, and the key question is which points can be blown up to ensure $-K_X$ stays big and nef. To this end, Demazure introduced the notion of  {\em points in (almost) general position}, a notion we need to generalise. 

Consider $s\leq 9$ points $q_1,\ldots,q_s$ ``in'' $\PP^2$. We will allow infinitely near points, by which we mean in reality, there is a sequence of blowups
\begin{equation} \label{eblowup} 
X = X_s \xrightarrow{f_s} X_{s-2} \lm \ldots \lm X_1 \xrightarrow{f_1}  X_0 :=\PP^2 
\end{equation}
and $q_i \in X_{i-1}$. Note that $K_X^2 =9 - s$, so when $s=9$, $X$ is not almost Fano although it may satisfy the hypotheses of Proposition~\ref{pnec}.

\begin{defn}  \label{dgenpos}  \mlabel{dgenpos}(\cite[II Definition~2.1]{Dem}, \cite[Definition~3.2]{HW})
Suppose that $s <9$. The points $q_1,\ldots,q_s$ are in {\em general (resp. almost general) position} if the following conditions all hold:
\begin{enumerate}
 \item no 3 (resp. 4) points lie on a line,
 \item no 6 (resp. 7) points lie on a conic,
 \item there are no infinitely near points (resp. no $q_i$ lies on a (-2)-curve) and,
 \item no singular cubic passes through 8 points in such a way that the singularity is one of those points (resp. no extra condition on 8 points). 
\end{enumerate}
\end{defn}

From \cite[II Th\'{e}or\`{e}me~1]{Dem}, we know that the only (smooth) Fano surfaces (also called {\em del Pezzo surfaces} in the literature) are $\PP^1 \times \PP^1$ and the blowups of $\PP^2$ at up to 8 points in general position. If one blows up 3 points which lie on a line $E$ in $\PP^2$, then the strict transform $\hat{E}$ of $E$ in the blowup $X$ will be a (-2)-curve so $-K_X.\hat{E} = 0$ and $X$ is not Fano. However, it is almost Fano. Blowing up 4 points on a line yields a (-3)-curve, so the resulting blowup no longer has $-K_X$ nef. Demazure showed similar phenomena occur when one violates the conditions in ii),iii) and iv) above. 

We now consider the case $s=9$ and examine the hypotheses of Proposition~\ref{pnec}. We may as well assume that any 8 are in almost general position, otherwise $-K_X$ is not nef. 

\begin{lemma}  \label{ldem}  \mlabel{ldem}
 Let $f: X \lm \PP^2$ be the blowup of $\PP^2$ at 9 points $q_1,\ldots,q_9$ (some possibly infinitely near). Then there exists an {\em anti-canonical curve} $C_X$, (that is, $C_X \in |-K_X|$) and $H^1(\ox(-rK_X)) = 0$ for all $r \geq 0$ if and only if $H^0(\calo_{C_X}(rK_X)) = 0$ for all $r > 0$. 
\end{lemma}
\begin{proof}

The existence of an anti-canonical divisor is well-known and can be deduced as follows. Note that $R^if_*(\ox) = 0$ for $i=1,2$ so by Riemann-Roch we have $\chi(\ox(-K_X)) = \chi(\ox) = \chi(\calo_{\PP^2}) = 1$. Serre duality ensures then that $H^0(\ox(-K_X)) \neq 0$ and we may find a divisor $C_X \in |-K_X|$. 

Now $X$ has negative Kodaira dimension, so Serre duality gives the following exact sequence
$$ H^1(\ox(-(r-1)K_X)) \lm H^1(\ox(-rK_X)) \lm H^1(\calo_{C_X}(-rK_X)) \lm 0.$$
Now $H^1(\ox) = H^1(\calo_{\PP^2}) = 0$ so by induction we see that $H^1(\ox(-rK_X)) = 0$ for all $r \geq 0$ if and only if $H^1(\calo_{C_X}(-rK_X)) = 0$ for all $r > 0$. Now $C_X$ is anti-canonical so $\omega_{C_X} \simeq \calo_{C_X}$ and Serre duality shows that $H^1(\calo_{C_X}(-rK_X)) = H^0(\calo_{C_X}(rK_X))^*$. 
\end{proof}\vs

We wish now to re-write the condition $H^0(\calo_{C_X}(rK_X)) = 0$ for all $r > 0$ in terms of a cubic $C$ passing through the 9 points. Note $C$ exists since $H^0(\calo_{\PP^2}(3)) = 10$, and it is furthermore anti-canonical. From this, we can construct an anti-canonical curve $C_X$ on $X$ with $f_*C_X = C$ as follows. We use the sequence of blowups \eqref{eblowup}. Let $C_0=C$ and suppose we have constructed inductively, an anti-canonical curve $C_{i-1} \subset X_{i-1}$. Let $f_{i*}^{-1}$ denote the strict transform, $E_i$ the exceptional curve of $f_i$ and $m_i$ the multiplicity of $C_{i-1}$ at $q_i$. Since $C$ passes through the $q_i$, we know that $m_i \geq 1$. The adjunction formula now ensures that 
$$ C_i := f_{i*}^{-1} C_{i-1} + (m_i -1) E_i$$
is anti-canonical and $C_X=C_9$ is the desired curve.  

Recall that an effective divisor $D$ on $X$ is said to be {\em numerically 1-connected} if for every {\em effective} decomposition $D = D_1 + D_2$ (that is, with $D_1, D_2$ effective divisors), we have $D_1. D_2 \geq 1$. 

\begin{lemma}  \label{lnumconn}  \mlabel{lnumconn}
The anti-canonical curve $C_X$ is numerically 1-connected. 
\end{lemma}
\begin{proof}
We prove by induction that $C_i$ is numerically 1-connected. A simple case by case check shows that $C_0$ is numerically 1-connected. (Alternatively, note that $C_0$ is big and nef so we may invoke \cite[Lemma~3.11(i)]{Reid}). Any effective decomposition of $C_i$ has the form $C_i = (f^*_i D + lE_i) + (f^*_i D' - (l-1) E_i)$ where $C_{i-1} = D + D'$ is an effective decomposition and $l \in \Z$. Now
$$(f^*_i D + lE_i).(f^*_i D' - (l-1) E_i) = D.D' +l(l-1) \geq D.D' \geq 1 $$
so we are done. 
\end{proof}
\vs

Note that $C_X$ has arithmetic genus 1, and in fact $\omega_{C_X} \simeq \calo_{C_X}$ since it is anti-canonical. Furthermore, the lemma shows that $H^0(\calo_{C_X}) = k$ by \cite[Lemma~3.11(ii)]{Reid} or \cite[Corollary~12.3]{BHPV}. Let $\Gamma$ be the set of irreducible components of $C_X$ and recall there is a {\em degree} map 
$$\deg: \Pic C_X \lm \Z^{\Gamma}: \call \mapsto (\deg \call|_{D})_{D \in \Gamma}.$$ 
We let $\Pic^0 C_X$ denote the kernel  of this map as usual. 
\begin{lemma}  \label{lKdeg0}  \mlabel{lKdeg0}
\begin{enumerate}
\item If $H^0(\calo_{C_X}(rK_X)) = 0$ for all $r >0$, then $\deg \calo_{C_X}(K_X) = 0$. 
\item  $\deg \calo_{C_X}(K_X)=0$ if and only if all components except possibly one of $C_X$ is a (-2)-curve.
\item If $\deg \calo_{C_X}(K_X) = 0$ then $H^0(\calo_{C_X}(rK_X)) \neq 0$ if and only if $\calo_{C_X}(rK_X) \simeq \calo_{C_X}$
\end{enumerate}
In particular, $H^0(\calo_{C_X}(rK_X)) = 0$ for all $r >0$ if and only if $\calo_{C_X}(K_X)$ is non-torsion in $\Pic^0 C_X$. 
\end{lemma}
\begin{proof}
(i) Note that the total degree of $\calo_{C_X}(K_X)$ is $K_X.C_X = -K_X^2 = 0$ so if $\deg \calo_{C_X}(K_X) \neq 0$, then there exists a component $C^0\subset C_X$ on which $K_X$ is ample. It follows that $H^0(\calo_{C_X}(rK_X)) \neq 0$ for $r \gg 0$. 

(ii) 
From the genus formula, we know any (-2)-curve $E$ has $K_X.E = 0$. Hence if all components except possibly one are (-2)-curves, then $\deg \calo_{C_X}(K_X)=0$. Conversely, suppose that $\deg \calo_{C_X}(K_X)=0$. Then all smooth rational components are (-2)-curves. There can be at most one component which is not smooth rational, namely the strict transform of $C$ when it is irreducible of arithmetic genus one.  

(iii) This follows from the fact that $C_X$ is 1-connected (Lemma~\ref{lnumconn}) and \cite[Lemma~3.11(iii)]{Reid} or \cite[Lemma~12.2]{BHPV}. 
\end{proof}
\vs

Ideally, we would like to turn the condition that $\calo_{C_X}(K_X)$ is torsion into a condition on $C$. It seems however, that such a condition does not admit a nice description. We do the best we can, and give a description which covers the ``generic'' case and is sufficient to give lots of interesting examples where it holds and where it fails. Before doing so, it is informative to note that
\begin{enumerate}
\item $\Pic C$ and $\Pic C_X$ are isomorphic to the multiplicative group $k^*$ if $C$ is a not necessarily irreducible nodal cubic.
\item $\Pic C$ and $\Pic C_X$ are isomorphic to the additive group $k$ if $C$ is non-reduced or a cuspidal cubic. 
\end{enumerate}
In the reduced case, this can be found in \cite[Exercise~II.6.9]{Har} and the non-reduced case follows from the work of Artin (see \cite[Sections~4.13, 4.14]{Reid}).

\begin{lemma}  \label{lsamePic}  \mlabel{lsamePic}
The morphism $f: C_X \lm C$ induces an isomorphism $f^* \colon \Pic^0 C \lm \Pic^0 C_X$.
\end{lemma}
\begin{proof}
This is proved by showing case by case that $f_i^*\colon \Pic^0 C_{i-1} \lm \Pic^0 C_i$ is an isomorphism. There are quite a few possibilities, all with similar proof, so we shall only illustrate it in the case $f_i$ blows up a cusp. In this case, $C_i = E \cup C'$ where $E$ is the exceptional curve and $C'$ is the strict transform of $C_{i-1}$. Also, $E$ and $C'$ intersect in a single point $q$ tangentially with multiplicity 2. The Leray-Serre spectral sequence gives the exact sequence
$$ 0 \lm H^1(f_{i*}\calo_{C_i}^*) \xrightarrow{\iota} H^1(\calo_{C_i}^*) \xrightarrow{d} H^0(R^1f_{i*}\calo_{C_i}^*)  \lm 0.$$
Now $E$ is smooth rational, so $H^0(R^1f_{i*}\calo_{C_i}^*) = \Z$ and $d$ gives the degree of a line bundle restricted to $E$. Hence $\Pic^0 C_i$ lies in the image of $\iota$. Also, $f_i^*$ is the composite of $\iota$ with the natural map 
$$H^1(\calo_{C_{i-1}}^*)  \lm H^1(f_{i*}\calo_{C_i}^*).  $$
It suffices now to prove that the natural inclusion $j\colon \calo_{C_{i-1}}  \hookrightarrow f_{i*}\calo_{C_i}$ is an isomorphism. Away from the cusp $f_i(q)$, $j$ is of course an isomorphism. Any local section of $f_{i*}\calo_{C_i}$ at $f_i(q)$ is given by a local section $\xi$ of $\calo_{C'}$ at $q$, and a global section $\alpha$ of $E$ which agrees with $\xi$ at $q$ to order 2. Now $\alpha$ must be a constant so we see that $\xi$ is in fact a local section of $\calo_{C_{i-1}}$. 
\end{proof}
\vs

The lemma is clear when all the $q_i$ lie on the smooth locus of $C$. In this case, $f\colon C_X \lm C$ is already an isomorphism and furthermore, identifying $C_X$ with $C$ using $f$ we find  $\calo_{C_X}(K_X) \simeq \calo_C(-3) \otimes_C \calo(\sum_{i=1}^9 q_i)$. This motivates the following otherwise unconventional 

\begin{defn}  \label{dKXonC}  \mlabel{dKXonC}
Let $q_1,\ldots, q_9$ be 9 points in $\PP^2$, any 8 of which are in almost general position. We say that $\calo_C(-3) \otimes_C \calo(\sum_{i=1}^9 q_i)$ is a {\em well-defined degree 0 line bundle} if $\calo_{C_X}(K_X)$ has degree zero and in this case we define $\calo_C(-3) \otimes_C \calo(\sum_{i=1}^9 q_i)$ to be the line bundle on $C$ corresponding to $\calo_{C_X}(K_X)$ under the isomorphism of Lemma~\ref{lsamePic}. We say $q_1,\ldots, q_9$ are in {\em almost general position} if furthermore, $\calo_C(-3) \otimes_C \calo(\sum_{i=1}^9 q_i)$ is not torsion in $\Pic^0 C$. 
\end{defn}

\begin{cor}  \label{calmostgen}  \mlabel{calmostgen}
If $X$ is the blowup of $\PP^2$ at 9 points, then $H^1(\ox(-rK_X)) = 0$ for all $r \geq 0$ if and only if the points are in almost general position. 
\end{cor}
If we pick $C \subset \PP^2$ to be a smooth or nodal cubic, then $\Pic^0 C$ has an infinite number of torsion points, although most are not torsion. It is then easy to construct examples of 9 points on $C$ which are in almost general position and another 9 which are not. 

The following result shows that one implication of Conjecture~\ref{cmain} is almost true in the non-weighted case. It is a mild extension of \cite[III - Th\'{e}or\`{e}me~1]{Dem}. 

\begin{thm}  \label{tdemazure}  \mlabel{tdemazure}
Let $X$ be a smooth projective surface. The following are equivalent. 
\begin{enumerate}
 \item $X$ is almost Fano or the blowup of $\PP^2$ at 9 points in almost general position. (Note these two are mutually exclusive). 
 \item $X$ is either $\PP^1\times \PP^1$, the second Hirzebruch surface $\mathbb{F}_2 = \PP_{\PP^1}(\calo \oplus \calo(-2))$ or a blowup of $\PP^2$ at up to 9 points in almost general position. 
\item $X$ has negative Kodaira dimension and $H^1(\ox(-rK_X)) = 0$ for $r \geq 0$.  
\end{enumerate}
\end{thm}
\begin{proof}
 Demazure proved this theorem when $X$ is assumed to be the blowup of $\PP^2$ at up to 8 arbitrary points so we need only check the theorem in the case of other surfaces. Now both $\mathbb{F}_2$ and $\PP^1 \times \PP^1$ are almost Fano and have $H^1(\ox(-rK_X)) = 0$ for $r \geq 0$ so ii) $\Longrightarrow$ iii) and i) by Corollary~\ref{calmostgen}. 

We now prove iii) $\Longrightarrow$ i). From Proposition~\ref{pnec}, we know that $-K_X$ is nef and $K_X^2\geq 0$. We may assume that $K_X^2 = 0$ in which case either $X$ is ruled over an elliptic curve, so $H^1(\ox) \neq 0$ or $X$ is rational. In this latter case, the $-K_X$ nef condition ensures that $X$ is the blowup of $\PP^2$ at 9 points so we are done by Corollary~\ref{calmostgen}. 

We finally prove i) $\Longrightarrow$ ii) and so can assume that $X$ is almost Fano.  Since $-K_X$ is big and nef, some multiple of it is effective. It is non-zero as $K_X^2 >0$ so the Kodaira dimension of $X$ must be negative. Furthermore, $K_X^2 >0$ ensures that $X$ must be rational and $-K_X$ nef means that its relatively minimal model is either $\PP^1 \times \PP^1, \PP^2$ or $\mathbb{F}_2$. We are now reduced to the situation that Demazure has already proved. 
\end{proof}\vs

Together with Corollary~\ref{cneccriterion}, this theorem proves Theorem~\ref{tnec}.

\vs
Unfortunately, we do not know how to classify almost Fano weighted projective surfaces, and it seems that there is no nice analoguous description of them. If we limit the possible weights on ($-s$)-curves, then the following is useful in constraining the possibilities. 

\begin{prop}  \label{pcentreaFano}  \mlabel{pcentreaFano}
Let $(X, \Delta = \sum_i (1 - \frac{1}{p_i})D_i)$ be an almost Fano weighted projective surface. Suppose that whenever $D_i$ is a smooth rational curve with self-intersection $D_i^2 = -s$ and $s \geq 3$ we have $2p_i < s$. Then $-K_X$ is nef and $K_X^2 \geq 0$. In particular, $X$ is either almost Fano, the blowup of $\PP^2$ at 9 points, or geometrically ruled over an elliptic curve. 
\end{prop}
\begin{proof}
 We first show that $-K_X$ is nef. Let $C$ be an irreducible curve. Now $-(K_X + \Delta)$ nef ensures that 
\begin{equation}  \label{ecentreaFano}
-K_X.C \geq \Delta.C 
\begin{cases}
 > C^2 & \text{if $C$ is weighted and $C^2 <0$}, \\
 \geq 0 & \text{else}
\end{cases}
\end{equation}
We are done unless $C = D_i$ for some $i$ and $D_i^2 = -s$ is negative. The genus formula and (\ref{ecentreaFano}) gives 
$$ -K_X .D_i = 2 - 2p_a(D_i) + D_i^2 > D_i^2 $$
so for the strict inequality to hold, we must have that $p_a(D_i)=0$, that is, $D_i$ is smooth rational. Then 
$$ 2 + D_i^2 = -K_X.D_i  \geq \Delta.D_i \geq (1 - \tfrac{1}{p_i})D_i^2.$$
Re-arranging gives $2p_i \geq - D_i^2$ contradicting our assumption on the weights. 

Now $-(K_X + \Delta)$ big and nef ensures that some multiple of it is effective, so the result just proved yields
$$ K_X.(K_X + \Delta) \geq 0 \Longrightarrow K_X^2 \geq -K_X . \Delta \geq 0.$$
\end{proof}
\vs

For example, if $X$ is the Hirzebruch surface $\mathbb{F}_s = \PP_{\PP^1}(\calo \oplus \calo(-s))$ where $s \geq 3$, then there is only one rational curve of negative self-intersection and one must weight this curve if $(X,\Delta)$ is to have a 2-hereditary tilting bundle.

\section{Some 2-hereditary tilting bundles on projective surfaces}

In this section we show that many almost Fano surfaces do indeed have 2-hereditary tilting bundles and that furthermore, their endomorphism algebras are tame.

\begin{prop}  \label{ptame}  \mlabel{ptame}
 Let $(X,\Delta)$ be a weighted projective variety of dimension $n$. Suppose that one of the following holds.
\begin{enumerate}
 \item $(X,\Delta)$ is Fano or,
 \item $(X,\Delta) = X$ is an almost Fano surface.
\end{enumerate}
Then for any $n$-hereditary tilting bundle $\calt$ on $(X,\Delta)$, the endomorphism algebra $\End \calt$ is tame.
\end{prop}
\begin{proof}
We use the criterion of Proposition~\ref{pnheredbundle} and examine the orbit algebra $\Pi = \bigoplus_i \Hom_A(\calt,\omega_A^{\otimes -i} \otimes_A \calt)$ where $A$ is the GL-order associated to $(X,\Delta)$. Let $r$ be sufficiently divisible so that $r\Delta$ is integral. Then $\omega_A^{\otimes -ir} = \ox(-ir(K_X+\Delta)) \otimes_X A$ so the anti-canonical ring $R = \bigoplus_{i\geq 0} H^0(\ox(-ir(K_X+\Delta)))$ lies in the centre of $\Pi$. 

If $(X,\Delta)$ is Fano then $\proj R = X$ and we can use the theory of Serre modules with respect to the polarisation $-r(K_X+\Delta)$. Then $\Pi$ is the finite direct sum of the Serre modules of the following coherent sheaves on $X$: $\shom_A(\calt,\omega_A^{\otimes-i} \otimes_A \calt)$ for $i = 0,1,\ldots, r-1$. It is thus finitely generated over $R$.

If $X$ is an almost Fano surface, then we may take $r=1$ and \cite[Expos\'{e} V]{Dem} the anti-canonical model $\bar{X}:= \proj R$ is Gorenstein Fano. Furthermore, there is a birational morphism $f:X \lm \bar{X}$ which contracts (-2)-curves only to rational double points. The ring $R$ is finitely generated being also
the anti-canonical ring on $\bar{X}$. We may thus use the theory of Serre modules on $\bar{X}$ with respect to the polarisation $-K_{\bar{X}}$. Now $f^*\omega_{\bar{X}} = \omega_X$ so $\Pi$ is just the Serre module of the coherent sheaf $f_* \send_X \calt$. We now see that in both cases, $\End \calt$ is tame. 
\end{proof}\vs

We seek tilting bundles on almost Fano surfaces which are {\em quasi-canonical}, by which we mean that they are direct sums of line bundles. Given any direct sum of line bundles $\calt = \bigoplus_{i\in I} \call_i$, we let 
$$\mathbb{E} = \mathbb{E}(\calt) := \{ \call_i^{-1} \otimes_X \call_j | i,j \in I\},$$
the set of indecomposable summands of $\send_X \calt$. It is closed under inverses. The first result gives a simple criterion for when a quasi-canonical tilting bundle $\calt$  on $X$ is 2-hereditary.

\begin{prop}  \label{pheredDP}  \mlabel{pheredDP}
Let $X$ be an almost Fano surface, so there exists a smooth elliptic curve $C \in |-K|$. Let $\calt$ be a quasi-canonical tilting bundle on $X$. Then in the notation above, $\calt$ is 2-hereditary if the following 2 conditions hold:
\begin{enumerate}
 \item the first Chern class $c_1(\call)$ of any $\call \in \mathbb{E}(\calt)$ satisfies $c_1(\call).K_X \leq K_X^2$ and,
 \item if $c_1(\call).K_X = K_X^2$ then $\call|_C \not\simeq \calo_C(K_X)$. 
\end{enumerate}
Suppose that $c_1(\call)-K_X$ is effective. Then automatically $c_1(\call).K_X \leq K_X^2$ and furthermore, $c_1(\call).K_X < K_X^2$ if either i) $X$ is Fano and $c_1(\call)-K_X$ is non-zero or ii),  $c_1(\call)-K_X$ is linearly equivalent to a non-zero sum of curves, not all of which are (-2)-curves. 
\end{prop}
\begin{proof}
 We show by induction on $r$, that $H^i(\call(-rK_X)) = 0$ for all $i>0, r\geq 0, \call \in \mathbb{E}$. This will prove that $\calt$ is 2-hereditary. The case $r=0$ is just the partial tilting condition so we now assume that $r>0$. We use the exact sequence
\begin{multline*} 
H^1(\call(-(r-1)K_X)) \lm H^1(\call(-rK_X)) \lm H^1(\call(-rK_X)|_C) \\ \lm H^2(\call(-(r-1)K_X)) \lm H^2(\call(-rK_X)) \lm 0 .
\end{multline*}
By induction, it suffices to prove that $H^1(\call(-rK_X)|_C) = 0$ for all $r>0, \call \in \mathbb{E}$. Now $\omega_C = \oc(K_X + C) = \oc$ so by Serre duality on $C$, it suffices to show that $H^0(\call^{-1}(rK_X)|_C) = 0$. To check this, note
$$ \deg \call^{-1}(rK_X)|_C = -c_1(\call).(-K_X) -rK_X^2 \leq 0$$
by condition~i) and the fact that $K_X^2 >0$. If this degree is negative, then we are done so we suppose that it is zero. Then we must have $r=1$ and $c_1(\call).K_X = K_X^2$. Now condition~ii) guarantees that $\call^{-1}(K_X)|_C \not\simeq \oc$ so again $H^0(\call^{-1}(K_X)|_C) = 0$. We have thus proved that $\calt$ is 2-hereditary. 

Note that $-K_X$ is nef so $c_1(\call).K_X \leq K_X^2$ when $c_1(\call) -K_X$ is effective. Strict inequality occurs in the Fano case since $-K_X$ is ample. In the almost Fano case, the anti-canonical model contracts all the curves $E$ with $-K_X.E=0$, and these are precisely the (-2)-curves on $X$.
\end{proof}\vs

If $X = \PP^1 \times \PP^1$, then the proposition shows that the commonly used tilting bundles 
\begin{equation}\label{eP1P1bundle}
\calt = \ox \oplus \ox(1,0) \oplus \ox(0,1) \oplus \ox(1,1) \ \text{and}\ \ox \oplus \ox(1,0) \oplus \ox(1,1) \oplus \ox(2,1) 
\end{equation}
are 2-hereditary. Indeed, in this case $\ox(-K_X) = \ox(2,2)$ so $c_1(\call) - K_X$ is effective for all $\call \in \mathbb{E}(\calt)$. 

We look now at the Hirzebruch surface $X = \mathbb{F}_2 = \PP_{\PP^1}(\calo \oplus \calo(-2))$. Recall that it is ruled, say via $\pi: X \lm \PP^1$ and that the relative tautological bundle $\calo_{X/\PP^1}(1) = \ox(C)$ where $C$ is the unique section with negative self-intersection $C^2=-2$. The Picard group is generated by $C$ and a fibre $F$ of $\pi$. Now $X$ is almost Fano but not Fano, so the following shows that our Conjecture~\ref{cmain} will not be true if we replace almost Fano with Fano.

\begin{prop}  \label{pF2}  \mlabel{pF2}
 The bundle $\calt = \ox \oplus \ox(F) \oplus \ox(C+2F) \oplus \oc(C+3F)$ is a 2-hereditary tilting bundle on $X = \mathbb{F}_2$. 
\end{prop}
\textbf{Remark} This is the same as King's tilting bundle for $\mathbb{F}_2$ \cite[Proposition~6.1, Section~8]{K}, though his notation is different. Its endomorphism algebra is also computed in \cite[p. 9, Case iii)]{K}.
\begin{proof}
Since we know $\calt$ is a tilting bundle, we need only check the conditions in Propositions~\ref{pheredDP}. 
Now $-K_X = 2C+4F$ and 
$$\mathbb{E}(\calt) = \{\ox, \ox(\pm F), \ox(\pm(C+F)), \ox(\pm(C+2F)), \ox(\pm(C+3F))\}$$ 
so for any $\call \in \mathbb{E}(\calt)$ we see that $c_1(\call) - K_X$ is effective and not a multiple of the unique (-2)-curve $C$. 
\end{proof}\vs

We turn now to the question of constructing 2-hereditary tilting bundles on blowups of $\PP^2$. To this end, let $q_1,\ldots,q_r \in \PP^2$ be $r$ points in general position and $f: X \lm \PP^2$ the blowup of $\PP^2$ at these points. Let $\mmm_i \triangleleft \ox$ be the ideal sheaf of $q_i$ and $E_i\subset X$ be the exceptional curve above $q_i$. We will use the abbreviated notation $E_{ij} = E_i + E_j, E_{ijl} = E_i + E_j + E_l$ etc. We let $H\subset X$ be the pullback of a line in $\PP^2$. We need some results on cohomology vanishing. 

\begin{lemma}  \label{lcohomBlP}  \mlabel{lcohomBlP}
If $\call$ is any of the line bundles below, then $H^1(\call) = H^2(\call) = 0$.
\begin{enumerate}
 \item $\call = \ox(E_{12\ldots s} - E_{s+1})$ where $s<r$,
 \item $\call = \ox(H - E_{1\ldots s})$ where $0 \leq s \leq 3$, 
 \item $\call = \ox(2H - E_{1\ldots s})$ where $0 \leq s \leq 6$, 
 \item $\call = \ox(E_{1\ldots s}-tH)$ where $t\leq 2$ 
\end{enumerate}
\end{lemma}
\begin{proof}
In all cases we have $R^1f_*\call = 0$ so $H^i(\call) = H^i(f_* \call)$. We only present the proof in case~iii) as the others use the same technique and are easier. Consider the exact sequence
$$ 0 \lm f_*\call = \calo_{\PP^2}(2)\mmm_1\mmm_2\ldots\mmm_s \lm \calo_{\PP^2}(2) \lm \oplus_{i=1}^s k(q_i) \lm 0$$
where $k(q_i)$ is the skyscraper sheaf at $q_i$. The associated long exact sequence in cohomology then shows that $H^2(f_*\call) = 0$ and 
$$H^1(f_*\call) = \text{coker}\left( \phi:H^0(\calo_{\PP^2}(2)) \lm \oplus_{i=1}^s k\right).$$
It suffices to show that the map $\phi$ above is surjective, for which it is no loss of generality in assuming that $s=6$, by adding points in general position if necessary. (An easy exercise shows this is indeed always possible). Suppose the statement is false so since $\dim H^0(\calo_{\PP^2}(2)) = 6$ we can find a non-zero $Q \in H^0(\calo_{\PP^2}(2))$ which lies in $\ker \phi$. This means that $Q$ defines a conic in $\PP^2$ which passes through the 6 points $q_1,\ldots, q_6$. This contradicts the assumption that the $q_i$ are in general position. 
\end{proof}\vs

The following theorem, in conjunction with Theorem~\ref{tnec}, shows that Conjecture~\ref{cmain} cannot be too far wrong in the non-weighted case. 
\begin{thm}  \label{tdpbundles} \mlabel{tdpbundles}
 Let $X$ be the blowup of $\PP^2$ at $r\leq 6$ points in general position. Then the bundle $\calt$ below is a 2-hereditary tilting bundle on $X$.
\begin{enumerate}
 \item If $r=1$ then $\calt = \ox(E_1) \oplus \ox(H) \oplus \ox(H+E_1) \oplus \ox(2H)$,
 \item if $r=2$ then $\calt = \ox(E_{12}) \oplus \ox(H+E_1) \oplus \ox(H+E_2) \oplus \ox(H+E_{12}) \oplus \ox(2H)$,
 \item if $3 \leq r \leq 6$ then 
\begin{multline*}
 \calt = \ox(E_{123}) \oplus \ox(E_{1234})\oplus \ldots  \oplus \ox(E_{123r}) \\ \oplus \ox(H+E_{23}) \oplus \ox(H+E_{13}) \oplus \ox(H+E_{12}) \oplus \ox(H+E_{123}) \oplus \ox(2H).
\end{multline*}
 \end{enumerate}
\end{thm}
\textbf{Remark} In the cases, $r=1,2,3$ these are (up to shift) the same as the tilting bundles given in \cite[Propositions~6.1,6.2]{K}. The endomorphism rings were computed in \cite[Section~6, cases~ii), iv),v)]{K}
\begin{proof}
 We only present the case $r=6$ as the others can essentially be extracted from this one and use the same technique. The partial tilting condition $H^i(\call) = 0$ for all $i>0, \call \in \mathbb{E}(\calt)$ is easily checked using Lemma~\ref{lcohomBlP}. To see that the summands of $\calt$ generate the derived category, we follow the argument in \cite[Proposition~2.2]{HP}. Note first that $D^b_c(X)$ is generated by $\ox,\ox(H),\ox(2H),\calo_{E_1}(-1), \ldots, \calo_{E_6}(-1)$ so it suffices to show that the category ${\mathsf C}$ generated by $\calt$ contains these. Now one easily sees $\calo_{E_i}(-1) \in \mathsf{C}$. For example, picking any non-zero map $\psi: \ox(H+E_{12}) \lm \ox(H+E_{123})$ we find $\calo_{E_3}(-1) = \coker \psi \in {\mathsf C}$ and similarly for the others. To show $\ox,\ox(H) \in {\mathsf C}$ we can look at kernels, for example as follows. Note that $\ox(H+E_{12})|_{E_2} = \calo_{E_2}(-1)$ so 
$${\mathsf C} \ni \ker( \ox(H+E_{12}) \lm \ox(H+E_{12})|_{E_2}) = \ox(H+E_1).$$
Restricting to $E_1$ and repeating the argument shows that $\ox(H) \in {\mathsf C}$. A similar argument works for $\ox$ so $\calt$ is indeed a tilting bundle. 

It remains now only to check the conditions in Proposition~\ref{pheredDP}. It is useful to partially order divisor classes on $X$ by $D \leq D'$ if and only if $D'-D$ is linearly equivalent to an effective divisor. We bound below $c_1(\call)$ where $\call \in \mathbb{E}(\calt)$. We first bound the first Chern classes of the summands of $\calt$. Note that the maximal elements here are $H+E_{123}$ and  $2H$. Indeed, we may pass a conic through any 4 points so $2H-E_{123s}$ is effective and similarly so is $H-E_{ij}$. The same reasoning shows that the unique minimal element is $E_{123}$. Hence 
\begin{equation}  \label{edpbundles}
c_1(\call) \geq \quad \text{either \ a)} -H \ \ \text{or b)}\ E_{123} -2H 
\end{equation}
 Now $K_X = -3H +  E_{1\ldots 6}$ so in both cases we find $c_1(\call).K_X = 3 = K_X^2$. Since $-K_X$ is ample, we have $c_1(\call).K_X < K_X^2$ in all other cases. We need only now verify condition~ii) of Proposition~\ref{pheredDP}. To this end, consider a cubic curve $C$ passing through $q_1,\ldots, q_6$. We consider case a) first and suppose to the contrary that $\oc(K_X) \simeq \oc(-H)$, or equivalently, that $\oc(2) \simeq \oc(q_1 + \ldots q_6)$. Now $H^1(\calo_{\PP^2}(-1)) = 0$ so $H^0(\calo_{\PP^2}(2)) \lm H^0(\oc(2))$ is surjective and we may find a non-zero $Q \in  H^0(\calo_{\PP^2}(2))$ which defines a conic passing through $q_1,\ldots, q_6$. This contradicts the fact that they are in general position. We may similarly dispose of case b) using the fact that $q_4,q_5,q_6$ are not collinear. This completes the proof of the theorem.  
\end{proof}\vs

We have looked at one last case which lends further evidence to our main conjecture.

\begin{prop} \label{pcollinear}  \mlabel{pcollinear}
 Let $X$ be the blowup of $\PP^2$ at 3 collinear points. Then 
$$\calt =  \ox(E_1) \oplus \ox(E_{12}) \oplus \ox(E_{13}) \oplus \ox(H) \oplus \ox(H+E_1) \oplus \ox(2H)$$
is a 2-hereditary tilting bundle on $X$
\end{prop}
\begin{proof}
As in the proof of Theorem~\ref{tdpbundles} we see that $\calt$ generates the derived category. To check the partial tilting condition, we need to compute $H^i(\call)$ as $\call$ ranges over $\mathbb{E}(\calt)$. In this case, the possible $\call$ involve at most two of the exceptional curves, so blowing down the remaining one via say $f:X \lm \bar{X}$ we see $H^i(\call) = H^i(f_* \call)$. We may now invoke  Lemma~\ref{lcohomBlP} on $\bar{X}$ to show $\calt$ is indeed a tilting bundle. It remains now to verify the conditions in Proposition~\ref{pheredDP}. Now $c_1(\call)$ is bounded below by $E_1 - 2H$ so $c_1(\call) - K_X$ is bounded below by 
$$ E_1 - 2H + 3H - E_{123} = H - E_{23}.$$
This is effective and not a multiple of the unique (-2)-curve $E$, since $E \sim H - E_{123}$. 
\end{proof}

\section{$n$-hereditary tilting bundles associated to group quotients}  \label{sgroup}  \mlabel{sgroup}

We now construct some $n$-hereditary tilting bundles on groups quotients using the skew group algebra. 

Let $\pi:\Xt \lm X$ be a (ramified) Galois cover of smooth projective varieties with finite Galois group $G$. By this we mean that $G$ acts faithfully on $\Xt$ and that $\pi$ exhibits the scheme-theoretic quotient of $\Xt$ by $G$. Now $\pi_*\oxt$ is a sheaf of algebras on $X$, and $G$ acts on this sheaf of algebras so we may form the skew group algebra $A = \pi_* \oxt \# G$. We will abuse notation and denote this by $\oxt \# G$. Note that this is even an order, for $A$ acts faithfully on the locally free sheaf $\pi_*\oxt$ so embeds in $\send_X \pi_* \oxt$. Furthermore, Maschke's theorem ensures that locally at any closed point of $X$, $A$ has global dimension $n = \dim \Xt$. 

These skew group algebras often give rise to GL-orders as follows. Let $D_1,\ldots, D_r \subset X$ be the ramification divisors and $p_1,\ldots,p_r$ be the corresponding ramification indices. We assume these are simple normal crossing so the natural weighted projective variety to associate to these data is $(X,\sum(1 - \frac{1}{p_i})D_i)$. We wish to show that under favourable circumstances, $\calo_{\Xt} \# G$ is Morita equivalent to the corresponding standard GL-order. To this end, let $\Dt_1,\ldots, \Dt_r$ be the reduced inverses images of the $D_i$ so that $\pi^* D_i = p_i\Dt_i$. Note that $\Dt_i$ is $G$-invariant so $\calo_{\Xt}(-\Dt_i)A$ defines a two-sided ideal of $A$ and we may view $\calo_{\Xt}(\Dt_i) \otimes_{\Xt} A$ as an invertible $A$-bimodule. Now $\oxt$ is a locally projective $A$-module, being a summand of $A$, so tensoring by these invertible bimodules gives many more. Recall from \cite[Proposition~6.7]{AZ}, that a {\em local progenerator} for $A$ is an $A$-module $\mathcal{P}$ such that on any affine open $U \subset X$, the module of sections $\mathcal{P}(U)$ is a progenerator for $A(U)$. As in the classical ring-theoretic Morita theory, we have that $A$ and $\send_A \mathcal{P}$ are Morita equivalent via $\shom_A(\mathcal{P}, -)$. In our case, the natural candidate for $\mathcal{P}$ is the locally projective $A$-module  
\begin{equation}  \label{eMoritabundle}
 \mathcal{P} = \bigoplus_{i=1}^r \bigoplus_{j_i = 0}^{p_i-1} \oxt(-\sum j_i \Dt_i) .
\end{equation}

Let $\Dt^0_i$ be a component of $\Dt_i$. There is a corresponding {\em inertia group}  
$$ H_i = \{ h \in G | h.x = x \ \text{for all } x \in \Dt^0_i  \}$$
and the inertia groups of other components of $\Dt_i$ are the conjugates of this $H_i$. 

Let $\tilde{q} \in \Xt$ and $H$ be the stabiliser $\text{Stab}_G (\tilde{q})$. Then $\oxt(-\Dt_i) \otimes_{\Xt} k(\tilde{q})$ is a 1-dimensional $H$-module so corresponds to a character of $H$. This character is trivial if $\tilde{q} \notin \Dt_i$. 

\begin{thm}  \label{tskewisGL}  \mlabel{tskewisGL}
Let $\mathcal{P}$ be the locally projective $A$-module in Equation~(\ref{eMoritabundle}). Then 
$\send_A \mathcal{P}$  is the standard GL-order $T_{p_1}(D_1) \otimes_X \ldots \otimes T_{p_r}(D_r)$. Suppose that for any singular point $\tilde{q}$ of $\cup \Dt_i$, the following holds:
\begin{itemize}
 \item[(*)] the stabiliser $H = \text{Stab}_G\, (\tilde{q})$ is abelian and its character group $H^*$ is generated by the characters $\oxt(-\Dt_i) \otimes_{\Xt} k(\tilde{q})$ for  $i = 1,\ldots, r$. 
\end{itemize}
Then $\mathcal{P}$ is a locally projective generator so the skew group algebra $A = \oxt\# G$ is a GL-order. 
\end{thm}
\textbf{Remark} In particular, if $\cup \Dt_i$ is smooth as happens when $X$ is a curve, we obtain a GL-order automatically. 
\begin{proof}
The identification of $\send_A \mathcal{P}$ with the standard GL-order is an easy exercise we omit using the fact that $\shom_A(-,-) = \shom_{\Xt}(-,-)^G$ and $\oxt(\sum j_i \Dt_i)^G = \ox(\lfloor \frac{j_i}{p_i}\rfloor D_i)$ for all integers $j_i$. 

Note that (*) holds if $\tilde{q}$ is an unramified point, for then $\text{Stab}_G (\tilde{q}) = 1$. We show it also holds at all smooth points $\tilde{q}$ of $\cup \Dt_i$. Indeed, suppose that $\tilde{q}$ lies in the component $\Dt_i^0$ of $\Dt_i$ and that $H_i$ is the inertia group of $\Dt_i^0$. It suffices to show that $H := \text{Stab}_G (\tilde{q}) = H_i$ for the inertia group is cyclic and acts on the conormal bundle of $\Dt_i$ by a character which generates $H_i^*$. We may factor 
$$\pi: \Xt \xlm{\pi'} \Xt/H_i \xlm{\pi''} X .$$
Now $H_i$ is the generic stabiliser of points on $\Dt_i^0$ so $\pi'': \Xt/H_i \lm X$ is generically unramified along $\pi'(\Dt_i^0)$. Purity of the branch locus ensures that it is unramified at $\pi'(\tilde{q})$ too. Let $\hat{R}$ be the complete local ring of $\Xt$ at $\tilde{q}$ and $R$ be the complete local ring of $X$ at $\pi(\tilde{q})$. Note that $H \supseteq H_i$ acts on $\hat{R}$ and in fact, $\hat{R}/R$ is Galois with Galois group $H$. Since $\hat{R}^{H_i}/R$ is unramified, we must have $H=H_i$. 

It suffices now to assume (*) holds for all closed points  $\tilde{q} \in \Xt$ and show that $\mathcal{P}$ is a generator in the sense that it surjects onto any simple $A$-module. Now any simple $A$-module $M$ has a central character determined by a closed point $q \in X$ and so is an $A \otimes_X k(q)$-module. Let $\tilde{q}\in \Xt$ be any point lying above $q$ and $Q = G.\tilde{q}$ be its $G$-orbit. We see that $A \otimes_X k(q) = \pi^* k(q) \# G$ has a nilpotent ideal generated by $\calo_{\Xt}(-Q)$ so $M$ is even a module over $\calo_Q \# G$. Now $\calo_Q$ is semisimple, being the direct product of $|G/H|$ copies of $k$ so $\calo_Q\# G$ is semisimple too. For each left coset $C$ of $H = \text{Stab}_G (\tilde{q})$ in $G$, there is a corresponding idempotent $e_C \in \calo_Q$ which is 1 at $C.\tilde{q}$ and 0 on the rest of $Q$. We may thus decompose $M$ as a direct sum of the components $e_CM$. Now $G$ permutes these components, so $\dim_k M \geq |G/H|$ and Wedderburn theory tells us there are at most $|H|$ non-isomorphic $\calo_Q\# G$-modules. It thus suffices to show that $\mathcal{P}$ surjects onto $|H|$ non-isomorphic simple modules supported at $q$. By construction, $\mathcal{P}$ surjects onto the simple modules of the form 
\begin{equation}  \label{esimples}
\oxt(-\sum_i j_i \Dt_i) \otimes_{\Xt} \calo_Q  
\end{equation}
Now $\oxt(-\Dt_i) \otimes_X -$ alters the $H$-module $e_H M$ by a character $\chi_i \in H^*$ and (*) ensures that the $\chi_i$ generate $H^*$. Hence there are at least $|H|$ non-isomorphic simples of the form (\ref{esimples}) and the theorem is proved.  
\end{proof}
\vs

Let $n$ be the dimension of $X$. We now look at the question of producing $n$-hereditary tilting bundles for the order $A = \oxt \# G$, or equivalently, the quotient stack $[\Xt/G]$. 
\begin{defn} \label{dGstable}  \mlabel{dGstable}
We say that a bundle $\calt$ on $\Xt$ is {\em $G$-stable} if $g^* \calt \simeq \calt$ for all $g \in G$. 
\end{defn}
We have an easy

\begin{prop}  \label{pGstabletilt}  \mlabel{pGstabletilt}
Let $\calt$ be a $G$-stable $n$-hereditary tilting bundle on $\Xt$. Then $A \otimes_{\Xt} \calt$ is an $n$-hereditary tilting bundle on $A$. 
\end{prop}
\begin{proof}
 Note first that
$$ \Ext^p_A(A \otimes_{\Xt} \calt, - ) = H^p(\shom_A(A \otimes_{\Xt} \calt, - )) = 
H^p(\shom_{\Xt}(\calt, - )) =\Ext^p_{\Xt}(\calt, - ).$$
Hence if $\calm$ is an $A$-module such that $\Ext^p_A(A \otimes_{\Xt} \calt, \calm) = 0$ for all $p$, then $\Ext^p_{\Xt}(\calt, \calm) = 0$ and we must have $\calm = 0$ since $\calt$ generates $D^b_c(\Xt)$. 

We now show the partial tilting and $n$-hereditary conditions. Note that $\omega_A = \omega_{\Xt} \otimes_{\Xt} A$ and that as a sheaf on $\Xt$ we have $A \otimes_{\Xt} \calt = \bigoplus_{g \in G} g^* \calt$. Hence for $p > 0$ we have 
$$ \Ext^p_A(A \otimes_{\Xt} \calt, \omega_A^{-r} \otimes_A A \otimes_{\Xt} \calt) = 
\Ext^p_{\Xt}(\calt, \bigoplus_{g \in G}\omega_{\Xt}^{-r} \otimes_{\Xt} g^* \calt) = 0 $$
since $\calt$ is a $G$-stable $n$-hereditary tilting bundle on $\Xt$. 
\end{proof}
\vs

This gives the following well-known fact. 
\begin{eg}  \label{eFanocurves}  \mlabel{eFanocurves}
The Fano weighted curves, weighted on 3 points have 1-hereditary tilting bundles.  Indeed, 
in this case, the weighted curves arise as GL-orders of the form $A = \calo_{\PP^1} \# G$ where $G$ is a finite subgroup of $PSL_2$ corresponding to types $D$ or $E$. The proposition applies then to the $G$-stable 1-hereditary tilting bundle $\calo \oplus \calo(1)$ on $\PP^1$  to give 1-hereditary tilting bundles on $A$. This example can be extended to the weighted projective line, weighted on 2 points with the same weight, but not to the other Fano weighted curves. 
\end{eg}

One usually prefers to deal {\em basic} tilting bundles, whereby we mean that the indecomposable summands are non-isomorphic. We give some results which help reduce the tilting bundle in Proposition~\ref{pGstabletilt} to a basic one. Let $\call$ be a bundle on $\Xt$. We will assume that $\call$ is {\em End-simple}, by which we mean that $\End_{\Xt} \call = k$ so in particular, $\call$ is indecomposable. 

\begin{defn}  \label{dstab}  \mlabel{dstab}
We define the {\em stabiliser of $\call$} to be 
$$\text{Stab}_G\, \call = \{ g \in G| g^* \call \simeq \call \}.$$
Given a subgroup $H$ of $G$, an {\em $H$-equivariant structure on $\call$} is an $\oxt\# H$-module structure on $\call$ that extends the $\oxt$-module structure on $\call$.
\end{defn}
Given a $G$-equivariant structure on $\call$, note that multiplication by $g \in G$ is a sheaf map $g^* \call \lm \call$. This makes the following result clear. 

\begin{prop}  \label{psameAtensor}  \mlabel{psameAtensor}
Let $\call$ be a sheaf on $\Xt$ and $ g \in G$. Then 
$$ A \otimes_{\Xt} \call \simeq A \otimes_{\Xt} g^* \call.$$
\end{prop}
Hence if $\calt = \oplus_{i \in I} \calt_i$ is the decomposition of a $G$-stable $n$-hereditary tilting bundle on $\Xt$ into indecomposable bundles, then $\calt' = \bigoplus_{j \in J} A \otimes_{\Xt} \calt_j$ is an $n$-hereditary tilting bundle on $A$, if $J\subseteq I$ includes a representative of each $G$-orbit. 

\begin{lemma} \label{lHstructures}  \mlabel{lHstructures}
Let $\call$ be an End-simple bundle on $\Xt$ and $H$ be an abelian subgroup of $\text{Stab}_G\, \call$. 
\begin{enumerate}
 \item If there exists an $H$-equivariant structure on $\call$, then up to isomorphism, there are exactly $|H|$ such structures, say $\call^{(1)}, \ldots, \call^{(|H|)}$. In this case 
\begin{equation} \label{eHstructures}
 (\oxt\# H) \otimes_{\Xt} \call \simeq \bigoplus_{1 \leq i \leq |H|} \call^{(i)} . 
\end{equation}
 \item If $H$ is cyclic, then there always exists an $H$-equivariant structure on $\call$.
 \item If $H = \text{Stab}_G\, \call$ and $\call^{(1)}$ is an $H$-equivariant structure on $\call$, then $A \otimes_{\oxt\# H} \call^{(1)}$ is an indecomposable $A$-module. 
\end{enumerate}
\end{lemma}
\begin{proof}
To simplify notation, we let $B = \oxt \# H$. Suppose that $\call^{(1)}$ is an $H$-equivariant structure on $\call$. For any character $\chi \in H^*$, $\call^{(1)} \otimes_k \chi$ is a $B$-module whose action is defined by $\oxt$ acting on $\call^{(1)}$ and $H$ acting diagonally on the tensor product. Hence there are at least $|H|$ $H$-equivariant structures on $\call$. Moreover, there are no others for any two $H$-equivariant structures can only differ in the $H$-module structure, which must be by a character as $\call$ is End-simple. 

We prove part ii) and assume that $H = \langle \h \rangle$ is a cyclic group of order $p$. Consider an arbitrary isomorphism $\phi: h^* \call \lm \call$ which we view as multiplication by $h$. It generates an $H$-equivariant structure on $\call$ if and only if the induced map $\phi^p:=\phi(h^*\phi) \ldots(h^{(p-1)*}\phi): \call = (h^p)^*\call \lm \call$ is the identity map. Now $\call$ is End-simple so $\phi^p$ is at worst a scalar multiple $\alpha$ of the identity and changing $\phi$ by a $p$-th root of $\alpha$ now gives an $H$-equivariant structure and establishes ii).

We now establish the isomorphism in (\ref{eHstructures}). Note first that 
$$ \Hom_{B}(B \otimes_{\Xt} \call, \call^{(i)}) = 
\Hom_{\Xt}(\call, \call^{(i)}) = k$$
where we have written =, since all isomorphisms are canonical. If for each $i$, we pick the homomorphism corresponding to 1, then we obtain an $\oxt\# H$-module homomorphism
$$  \Psi: B \otimes_{\Xt} \call \lm \bigoplus_{1 \leq i \leq |H|} \call^{(i)} . $$
To show this is an isomorphism, we will view this as a morphism $\Psi': \call^{|H|} \lm \call^{|H^*|}$ as follows. Firstly, the $H$-equivariant structure on $\call^{(1)}$ allows us to identify each summand $h^*\call, h \in H$ of $B \otimes_{\Xt} \call$ with $\call$. On the other hand, we have seen that $\call^{(i)}\simeq \call^{(1)} \otimes_k \chi$  for some $\chi\in H^*$ so the codomain of $\Psi$ is naturally a direct sum of copies of $\call$ indexed by the characters of $H$. Hence $\Psi'$ is given by a matrix which is the character table of $H$. This is invertible so $\Psi$ is the desired isomorphism.

We now prove part~iii). As a sheaf on $\Xt$, we know that $A \otimes_{\oxt\# H} \call^{(1)}$ is the direct sum of $|G/H|$ non-isomorphic indecomposables $g^* \call$, where $g$ runs through a left transversal of $H$ in $G$. By Krull-Schmidt, any non-zero $A$-module summand $M$ must contain an $\oxt$-module summand $M'$ isomorphic to $g^* \call$ for some $g \in G$. Now multiplication by $g' \in G$ is skew-$\oxt$-linear so $g'M'$ is also a summand of $M$. Hence $M$ is isomorphic to the direct sum of the same $|G/H|$ summands as $A \otimes_{\oxt\# H} \call^{(1)}$ so equals $A \otimes_{\oxt\# H} \call^{(1)}$.
\end{proof}
\vs 

Suppose $H$ is a normal subgroup of $G$. Then conjugation by $g \in G$ induces an automorphism $c_g$ of $B:= \oxt \# H$. Thus given a $B$-module $\calm$, we may define $g^* \calm = B \otimes_{c_g,B} \calm$ where the subscript on the tensor means we use the ring homomorphism $c_g: B \lm B$. 

\begin{thm}  \label{tfindbasic}  \mlabel{tfindbasic}
\begin{enumerate}
\item Let $\call,\call'$ be indecomposable bundles on $\Xt$. Then the indecomposable $A$-module summands of $A \otimes_{\Xt} \call$ and $A \otimes_{\Xt} \call'$ are non-isomorphic unless $\call' \simeq g^* \call$ for some $g \in G$.
\item Let $\call$ be an End-simple bundle on $\Xt$ with stabiliser $H = \text{Stab}_G\, \call$. Suppose that $H$ is a normal abelian subgroup of $G$ and that there is an $H$-equivariant structure $\call^{(1)}$ on $\call$. Then 
$$ A \otimes_{\Xt} \call \simeq \bigoplus_{\chi \in H^*} A \otimes_B \call^{(1)} \otimes_k \chi$$
is the decomposition of $A \otimes_{\Xt} \call$ into non-isomorphic  indecomposable $A$-module summands. 
\end{enumerate}
\end{thm}
\begin{proof}
We prove part~i) first. If $A \otimes_{\Xt} \call$ and $A \otimes_{\Xt} \call'$ have an  isomorphic $A$-module summand, then they have an isomorphic indecomposable $\oxt$-module summand, say $g^* \call$. It follows that as $\oxt$-modules, we have $A \otimes_{\Xt} \call' \simeq \oplus_{g \in G} g^* \call$ and i) holds.

We now prove ii). Lemma~\ref{lHstructures} gives the desired decomposition into indecomposables. It remains only to show that the summands are non-isomorphic. Suppose then that $A \otimes_B \call^{(1)} \simeq A \otimes_B \call^{(1)} \otimes_k \chi$. Decomposing both sides as $B$-modules, we see that there exists some $B$-module isomorphism $\phi:\call^{(1)} \otimes_k \chi \simeq g^* \call^{(1)}$ for some $g \in G$. Now $\phi$ is in particular, an isomorphism of $\oxt$-modules so $g \in \text{Stab}_G\, \call = H$. But then $g^* \call^{(1)} \simeq \call^{(1)}$ as conjugation by $g$ on $B$ is inner. Lemma~\ref{lHstructures}i) then shows that $\chi$ is trivial and the theorem is proved. 
\end{proof}
\vs

We now look at some examples in the case of surfaces. Unfortunately, there seems to be very few as Galois covers of smooth surfaces tend not to be smooth (only normal). 

\begin{eg}  \label{econicram}  \mlabel{econicram}
Let $\Xt = \PP^1 \times \PP^1$. Then the cyclic group of order 2, $G = \{1,g\}$ acts on $\Xt$ by $g.(a,b) = (b,a)$. The quotient is $X = \Xt/G = \PP^2$. The ramification divisor on $\Xt$ is the (1,1)-divisor defined by $a=b$, and its image in $\PP^2$ is a smooth conic $C$. Hence, Theorem~\ref{tskewisGL} ensures that $A = \oxt \# G$ is a GL-order associated to $(\PP^2,\frac{1}{2}C)$.

We may also blow up $\Xt$ at a generic $G$-orbit, say $\{(0,\infty),(\infty,0)\}$ to obtain $f:\Xt_1 \lm \Xt$. Let the exceptional curves be $E,E'$. Now $G$ also acts on $\Xt_1$ and the quotient $X_1$ is the blowup of $X = \PP^2$ at the one point $f(0,\infty)$. Again, $A_1 = \oxt_1 \# G$ is the GL-order associated to $(X_1,\frac{1}{2}f^{-1}C)$ where we have used $f$ denote the blowup $X_1 \lm X$ as well.
\end{eg}

\begin{cor}  \label{cconicram}   \mlabel{cconicram}
Using the notation outlined in \ref{econicram}, we have 
\begin{enumerate}
 \item $\calt = A  \oplus (A \otimes_{\Xt} \oxt(1,0)) \oplus (A \otimes_{\Xt} \oxt(1,1))$ is a basic 2-hereditary tilting bundle on $A$. The summand $A \otimes_{\Xt} \oxt(1,0)$ is an indecomposable rank 2 bundle , whilst the other two are direct sums of two line bundles.
 \item $\calt_1 = A_1  \oplus (A_1 \otimes_{\Xt_1} f^*\oxt(1,0)) \oplus (A_1 \otimes_{\Xt_1} \calo_{\Xt_1}(E)) \oplus (A_1 \otimes_{\Xt_1} f^*\oxt(1,1))$  is a basic 2-hereditary tilting bundle on $A_1$. The summands $A_1 \otimes_{\Xt_1} f^*\oxt(1,0), \ A_1 \otimes_{\Xt_1} \calo_{\Xt_1}(E)$ are indecomposable rank 2 bundles, whilst the other two are direct sums of two line bundles.
\end{enumerate}
\end{cor}
\begin{proof}
We prove ii) only, as i) is essentially a sub-case of ii). By Theorem~\ref{tfindbasic} and Proposition~\ref{pGstabletilt}, we need only show that 
$$\calt_1' = \calo_{\Xt_1} \oplus f^*\oxt(1,0) \oplus f^* \oxt(0,1) \oplus \calo_{\Xt_1}(E) \oplus \calo_{\Xt_1}(E') \oplus f^*\oxt(1,1)$$
is a $G$-stable 2-hereditary tilting bundle on $\Xt_1$. Now it is certainly $G$-stable and generates the derived category. The partial tilting condition can be checked directly using the Leray spectral sequence for $f: \Xt_1 \lm X_1$, or observing that $\Xt_1$ is $\PP^2$ with 3 general points blown up and using Lemma~\ref{lcohomBlP} to check cohomology vanishes. It remains only to check $\calt_1'$ is 2-hereditary, for which we may use Proposition~\ref{pheredDP} since $\calt'_1$ is quasi-canonical. Now given $\call \in \mathbb{E}(\calt'_1)$, $c_1(\call)$ is bounded below by $c_1f^*\calo_{\Xt}(-1,-1)$. Hence $c_1(\call) - K_{\Xt_1}$ is always non-zero effective and we are done.
\end{proof}
\vs

\begin{eg}  \label{econicramalg}
When the group $G$ is cyclic, or more generally, the hypotheses of Theorem~\ref{tfindbasic} hold for all the indecomposable summands of a tilting bundle $\calt$ on $\Xt$, then it is easy to calculate the the endomorphism ring of the tilting bundle $A \otimes_{\Xt} \calt$. We illustrate by computing the endomorphism rings in the case of Example~\ref{econicram}. Below we let $\chi_-$ be the non-trivial character of $G=\{1,g\}$. 

We start with the case $\Xt = \PP^1 \times \PP^1$. As a sheaf on $\Xt$,  $\calt = A  \oplus (A \otimes_{\Xt} \oxt(1,0)) \oplus (A \otimes_{\Xt} \oxt(1,1))$ is a direct sum of line bundles so maps between indecomposable summands can be written down as matrices with entries in the Cox ring $R$ of $\Xt$. Now $R$ is generated by the bases $x,y \in H^0(\oxt(1,0))$ and $x',y' \in H^0(\oxt(0,1))$. Furthermore, we may lift the action of $G$ to $R$ and assume that $x' = g.x, y' = g.y$. This gives a $G$-equivariant structure on $\oxt(1,1)$ which we consider the default, and $\oxt(1,1) \otimes_k \chi_-$ is the other one. There is always a natural $G$-equivariant structure on $\oxt$, and the other one is $\oxt \otimes_k \chi_-$. Now 
$$\Hom_A(\oxt,A \otimes_{\Xt} \oxt(1,0)) = \Hom_{\Xt}(\oxt, A \otimes_{\Xt} \oxt(1,0))^G = 
\Hom_{\Xt}(\oxt, \oxt(1,0) \oplus \oxt(0,1))^G$$ 
so has basis $\bfx = ( x \ x'), \bfy = (y \ y')$. Similarly, $\Hom_A(\ox\otimes_k \chi_-,A \otimes_{\Xt} \oxt(1,0))$ has basis $\bfx_- = ( x \ -x'), \bfy_- = (y \ -y')$. Given a vector $\mathbf{z} = (z \ z') \in R^2$, we define its ``adjoint'' to be $\mathbf{z}^{\dagger} := (g.z \ g.z')^T$. The quiver of $\End_A \calt$ is given by
$$\diagram
  & \oxt \otimes_k \chi_- \dto<.5ex>^{\bfy_-} \dto<-.5ex>_{\bfx_-}& \\
\oxt \rto<.5ex>^(.3){\bfx} \rto<-.5ex>_(.3){\bfy}& A \otimes_{\Xt} \oxt(1,0) 
\dto<.5ex>^{\bfy^{\dagger}_-} \dto<-.5ex>_{\bfx^{\dagger}_-}
\rto<.5ex>^(.6){\bfx^{\dagger}} \rto<-.5ex>_(.6){\bfy^{\dagger}}& \oxt(1,1)\\
  & \oxt(1,1) \otimes_k \chi_- & 
\enddiagram$$
The relations can easily be computed by matrix multiplication and using the relations in the Cox ring. We have some of the symmetric algebra type:
$$ \bfx \bfy^{\dagger} = \bfy \bfx^{\dagger}, \quad \bfx_- \bfy^{\dagger}_- = \bfy_- \bfx^{\dagger}_- .$$
The others are of exterior algebra type:
\begin{gather*}
 0 = \bfx \bfx_-^{\dagger} = \bfy \bfy_-^{\dagger} = \bfx \bfy_-^{\dagger} +  \bfy \bfx_-^{\dagger} \\
 0 = \bfx_- \bfx^{\dagger} = \bfy_- \bfy^{\dagger} = \bfx_- \bfy^{\dagger} +  \bfy_- \bfx^{\dagger}
\end{gather*}

We now look at the case where $\Xt$ is the blowup of $\PP^1 \times \PP^1$ at $q = (0 , \infty), q' = (\infty,0)$. We drop the subscript 1 from $\Xt$ to disencumber notation and for similar reasons, write $\oxt(m,n)$ for $f^*\calo_{\PP^1 \times\PP^1}(m,n)$. Note that the Cox ring $R$ of $\PP^1 \times \PP^1$ can be viewed as a subaglebra of the Cox ring $R_1$ of $\Xt$ via the pullback $f^*$. We will re-use the notation for elements of $R$ and may further assume that $x=0$ at $q$ and $y=0$ on $q'$. Now $x,y'$ are both 0 at $q$, so in $R_1$, we may factor $x,y'$ into
$$
 x\colon \oxt \xrightarrow{t} \oxt(E) \xrightarrow{\tfrac{x}{t}} \oxt(1,0), \quad 
y'\colon \oxt \xrightarrow{t} \oxt(E) \xrightarrow{\tfrac{y'}{t}} \oxt(0,1)$$
where $\tfrac{x}{t}, \tfrac{y'}{t}$ are formal symbols chosen to remind us of the factorisation above. Applying $g$ to these factorisations gives factorisations
$$
 x'\colon \oxt \xrightarrow{t'} \oxt(E') \xrightarrow{\tfrac{x'}{t'}} \oxt(0,1), \quad 
y \colon \oxt \xrightarrow{t'} \oxt(E') \xrightarrow{\tfrac{y}{t'}} \oxt(1,0)$$
The quiver of $\End_A \calt$ is now
$$\diagram
 & \oxt \otimes_k\, \chi_-  \dto^{\mathbf{t}_-}& & \\
 \oxt \rto^(.3){\mathbf{t}}& A \otimes_{\Xt} \oxt(E) \rto<.5ex>^{\mathbf{X}} \rto<-.5ex>_{\mathbf{Y}}& A \otimes_{\Xt} \oxt(1,0) 
\dto<.5ex>^{\bfy^{\dagger}_-} \dto<-.5ex>_{\bfx^{\dagger}_-}
\rto<.5ex>^(.6){\bfx^{\dagger}} \rto<-.5ex>_(.6){\bfy^{\dagger}}& \oxt(1,1) \\
 & & \oxt(1,1) \otimes_k \chi_- & 
\enddiagram$$
where $\bfx^{\dagger},\bfy^{\dagger},\bfx^{\dagger}_-,\bfy^{\dagger}_-$ are as in the $\PP^1 \times \PP^1$ case, and 
\begin{equation*}
 \mathbf{t} = (t \ t'), \quad \mathbf{t}_- = (t \ -t'), \quad
\mathbf{X} = 
\begin{pmatrix}
 \tfrac{x}{t} & 0 \\ 
 0 & \tfrac{x'}{t'}
\end{pmatrix}
\quad\mathbf{Y} = 
\begin{pmatrix}
  0 &  \tfrac{y'}{t}\\ 
 \tfrac{y}{t'} & 0
\end{pmatrix}
\end{equation*}
The relations now are
$$  \mathbf{X} \bfy^{\dagger} = \mathbf{Y} \bfx^{\dagger}, \quad \mathbf{X} \bfy^{\dagger}_- = - \mathbf{Y} \bfx^{\dagger}_-, \quad 
0 = \mathbf{t}\mathbf{X}\bfx^{\dagger}_- = \mathbf{t}\mathbf{Y}\bfy^{\dagger}_- =
\mathbf{t}_-\mathbf{X}\bfx^{\dagger}  = \mathbf{t}_-\mathbf{X}\bfx^{\dagger} $$
\end{eg}

\begin{eg}  \label{epolygonram}  \mlabel{epolygonram}
Let $G = \Z/p \times \Z/p$ say with generators $g_1,g_2$. Let $\zeta$ be a primitive $p$-th root of unity. Now $G$ acts on $\Xt = \PP^2_{x:y:z}$ by $g_1:(x:y:z) \mapsto (\zeta x:y:z), g_2:(x:y:z) \mapsto (x:\zeta y:z)$. The quotient $\Xt/G \simeq \PP^2_{u:v:w} =: X$. The ramification divisors on $X$ are the coordinate lines $u=0, v=0, w=0$ and the ramification indices here are all $p$. We will not dwell on this case as the canonical tilting bundle on the corresponding weighted projective plane is 2-hereditary. Let $\Xt_i, i =1,2,3$ be the blowup of $\Xt$ at $i$ of the coordinate points $(1:0:0),(0:1:0)$ or $(0:0:1)$, and let $E_1,\ldots,E_i$ be the corresponding exceptional curves. Since $G$ fixes the coordinate points, it acts on $\Xt_i$ and the quotient is easily seen to be the blowup $X_i$ of $\PP^2_{u:v:w}$ at $i$ coordinate points. The ramification divisor $D$ on $X_i$ consists of $i+3$ lines arranged in a polygon, namely, the $i$ exceptional divisors together with the strict transforms of the coordinate lines. Let $H \subset \Xt_i$ be the pullback of the generic line and let $A_i = \calo_{\Xt_i} \# G$ as usual. An explicit computation checking the condition (*) in Theorem~\ref{tskewisGL} shows that $A_i$ is a GL-order associated to the weighted projective surface $(X_i, \frac{1}{p}D)$. 
\end{eg}

\begin{cor}  \label{cpolygonram}  \mlabel{cpolygonram}
We use the notation in \ref{epolygonram}. Let $\calt$ be the 2-hereditary tilting bundle on $\Xt_i$ given in Theorem~\ref{tdpbundles}. Then $A_i \otimes_{\Xt_i} \calt$ is a basic 2-hereditary tilting bundle on $A_i$. It is the direct sum of $p^2(i+3)$ line bundles.
\end{cor}
\begin{proof}
First note that all line bundles on $\Xt_i$ are $G$-stable, and in fact, they all can be given $G$-equivariant structures. Indeed, if $D$ is a curve which $G$ leaves invariant (that is, $g(D) = D$ for all $g \in G$), then the action of $G$ on the constant sheaf of rational funtions $\mathcal{K}$ restricts to a $G$-action on $\calo_{\Xt_i}(D)$. Now the exceptional curves $E_j$ and the strict transform $H'$ of a coordinate line are all $G$-invariant. These generate the Picard group so every line bundle has a $G$-equivariant structure. The rest of the corollary now follows from  Proposition~\ref{pGstabletilt} and Theorem~\ref{tfindbasic}. 
\end{proof}
\vs

Although we will not compute explicitly the associated 2-hereditary algebras here, we will describe the general method for the computation. Suppose that $G$ is an abelian group acting on a smooth projective variety $\Xt$ and $D,D'$ are $G$-invariant divisors. Note that the tilting bundle $\calt$ in Corollary~\ref{cpolygonram} is a direct sum of line bundles of $\oxt(D)$ with $D$ a $G$-invariant divisor. It is easy to calculate $\Hom_A(A \otimes_{\Xt} \oxt(D), A \otimes_{\Xt} \oxt(D')$. Indeed in this case, $G$ acts naturally on $k(X)$ and so induces a $G$-equivariant structure on $\oxt(D),\oxt(D')$ and hence, $G$-action on $\Hom_{\Xt}(\oxt(D),\oxt(D'))$. Theorem~\ref{tfindbasic} then shows that $\Hom_A(A \otimes_{\Xt} \oxt(D), A \otimes_{\Xt} \oxt(D')$ is the direct sum of 
\begin{multline*}
\Hom_A( \oxt(D) \otimes_k \chi, \oxt(D') \otimes_k \chi') = 
\Hom_{\Xt}( \oxt(D) \otimes_k \chi, \oxt(D') \otimes_k \chi')^G 
\\ 
= \left(\Hom_{\Xt}( \oxt(D) , \oxt(D') ) \otimes_k \chi^{-1} \otimes_k \chi'\right)^G\end{multline*}
as $\chi,\chi'$ range over the character group $G^*$. This explains neatly, why the relations in Example~\ref{econicramalg} divide up into symmetric algebra type relations and exterior algebra type relations. They correspond to the decomposition of $H^0(\calo_{\Xt}(1,1))$ into symmetric and anti-symmetric tensors. 

\section{Generating the derived category}  \label{sgen}  \mlabel{sgen}

Lerner and Oppermann give a set of generators for the derived category of a weighted projective variety \cite[(proof of) Theorem~1.2]{LO}. However, the result is not explicitly stated and the description is more algebraic than geometric. In this section, we give a more constructive proof of the result using the Koszul complex, and present it in our geometric notation.

Let $A$ be a standard GL-order on $X$ associated to the weighted projective variety $(X,\sum_{i=1}^r (1 - \tfrac{1}{p_i})D_i)$. We need some notation to describe the generators. Let $I \subseteq \{1,\ldots, r\}$. We define $D_I := \cap_{i \in I} D_i$ and 
$$ \calo_{D_I} := \coker\left( \phi:\bigoplus_{i \in I} \calo_A(-\tfrac{1}{p_i}D_i) \lm \calo_A \right)$$
where $\phi$ is induced by inclusion of line bundles. An elementary explicit computation shows that as a sheaf on $X$, $\calo_{D_I}$ is just the structure sheaf on $D_I$ so there is no clash in notation and in particular, it is zero if and only if $D_I$ is empty. 

We have the following {\em Koszul resolution} of $\calo_{D_I}$ which is standard. 

\begin{lemma}  \label{lkoszul}  \mlabel{lkoszul}
There exists a chain complex $K^I_{\bullet}$ and a quasi-isomorphism $K^I_{\bullet} \lm \calo_{D_I}$ whose terms are given by 
$$ K^I_{m} := \bigoplus_{J \subseteq I, |J| = m} \calo_A(-\sum_{j \in J} \tfrac{1}{p_j}D_j).$$
\end{lemma}
\begin{proof}
Probably the easiest standard proof which applies in this setting is to use induction on $|I|$. If $j \notin I$, then $K^{I \cup\{j\}}_{\bullet}$ is just the total complex of the bicomplex $K^I(-\tfrac{1}{p_j}D_j)_{\bullet} \lm K^I_{\bullet}$. By induction, it is quasi-isomorphic to the complex 
$$ 0 \lm \calo_{D_I}(-\tfrac{1}{p_j}D_j) \xrightarrow{\psi} \calo_{D_I} \lm 0.$$
Now $\psi$ is injective by local computation. The lemma follows. 
\end{proof}
\vs

We omit the easy proof of the following. 

\begin{lemma}  \label{lgenerators}  \mlabel{lgenerators}
 Let $\call_1,  \call_2, \ldots$ be a set of generators for $\Coh X$. Then $A \otimes_X \call_1, A \otimes_X \call_2,\ldots$ is a set of generators for $A-\mo$. 
\end{lemma}

Before reformulating Lerner-Oppermann's result, we note that any $A$-module $\calm$ can be naturally viewed as an $(A,\ox)$-bimodule, so given a coherent sheaf $\call$ on $X$, $\calm \otimes_X \call$ is naturally an $A$-module too. Let $\text{Pic}^I_{[0,\Delta)}$ be the set of $I$-tuples $\alpha_I = (\alpha_i)_{i \in I}$ with $\alpha_i \in \{0,\frac{1}{p_i},\ldots, 1 - \frac{2}{p_i}\}$. The notation is motivated by the fact that 
$$\alpha_I D^I := \sum_{i \in I} \alpha_i D_i $$
gives the line bundle $\calo_A(\alpha_ID^I)$. Similarly, we define $\text{Pic}^I_{(0,\Delta]}$ be the set of $I$-tuples $\alpha_I = (\alpha_i)_{i \in I}$ with $\alpha_i \in \{\frac{1}{p_i},\ldots, 1 - \frac{1}{p_i}\}$. In this notation, we have the line bundle decomposition of 
\begin{equation}\label{edecomposeA}
A  = \bigoplus_{I \subseteq \{1,\ldots, r\}} \bigoplus_{\alpha_I \in \text{Pic}^I_{(0,\Delta]}} \calo_A(-\alpha_ID^I). 
\end{equation}

\begin{thm}  \label{tgen}  \mlabel{tgen}
For each $I \subseteq \{1,\ldots, r\}$ with $D_I \neq \varnothing$, let $\calt_I \in \Coh D_I$ generate the derived category of $D_I$. Then the set
$$ \Gamma:= \left\{ \calo_{D_I}(-\alpha_ID^I)\otimes_X \calt_I | I \subseteq \{1,\ldots, r\}, D_I \neq \varnothing, \alpha_I \in \text{Pic}^I_{[0,\Delta)} \right\}$$
generates the derived category of $A$. 
\end{thm}
\begin{proof}
Let $\calm$ be an $A$-module with $\Ext^m_A(\calo_{D_I}(-\alpha_ID^I)\otimes_X \calt_I, \calm) = 0$ for all $m,I,\alpha_I\in \text{Pic}^I_{[0,\Delta)}$. We wish to show that $\calm=0$ for which it suffices to show that 
$$\,^{\perp}\calm := \{ \mathcal{N} \in D^b_c(A) | \mathbf{R}\Hom_A(\mathcal{N}, \calm) = 0 \} $$
contains a set of generators for $A-\mo$. From Lemma~\ref{lgenerators} and Equation~(\ref{edecomposeA}), it suffices to show 
\begin{claim}  \label{ccontainsgen}  \mlabel{ccontainsgen}
$\,^{\perp}\calm$ contains $\calo_A(-\alpha_ID^I)\otimes_X \call$ for any line bundle $\call$ on $X$, $I \subseteq \{1,\ldots, r\}$ and $\alpha_I\in \text{Pic}^I_{(0,\Delta]}$. 
\end{claim}
By assumption, $\,^{\perp}\calm$ contains $\Gamma$. We first prove 
\begin{lemma} \label{lgenDI}  \mlabel{lgenDI}
$\,^{\perp}\calm$ contains $\calo_{D_I}(-\alpha_ID^I)\otimes_X \call$ for any line bundle $\call$ on $X$, $I \subseteq \{1,\ldots, r\}$ and $\alpha_I\in \text{Pic}^I_{[0,\Delta)}$.
\end{lemma}
\begin{proof}
Note that $\,^{\perp}\calm$ is closed under cones and contains $\calo_{D_I}(-\alpha_ID^I)\otimes_X \calt_I$. Now as a sheaf on $X$, $\calo_{D_I}(-\alpha_ID^I)$ is a line bundle on $D_I$ so $\calo_{D_I}(-\alpha_ID^I) \otimes_X (-)$ is an exact functor on $\Coh D_I$. Since $\calt_I$ generates the derived category of $D_I$, the lemma follows. 
\end{proof}

We continue with the proof of the claim, and hence theorem, by induction on $\sum \alpha_i$. The Koszul resolution from Lemma~\ref{lkoszul} gives an exact sequence
$$ K^I_{\bullet}(-\alpha_ID^I + \sum_{i \in I} \frac{1}{p_i} D_i) \otimes_X \call \lm \calo_{D_I} (-\alpha_ID^I + \sum_{i \in I} \frac{1}{p_i} D_i) \otimes_X \call \lm 0 $$
Now the left most term in homological degree $|I|$ is $\calo_A(-\alpha_ID^I)\otimes_X \call$ whilst the right most term lies in $\,^{\perp}\calm$ by Lemma~\ref{lgenDI}. Furthermore, the middle terms lie in $\,^{\perp}\calm$ by induction so we are done.  
\end{proof}

\section{Examples of 2-hereditary tilting bundles on weighted projective surfaces}  \label{srandom}  \mlabel{srandom}

In this section, we give examples of 2-hereditary tilting bundles on some weighted projective surfaces.

\begin{eg} \label{e11ram}  \mlabel{e11ram}
Let $X = \PP^1 \times \PP^1$ and $D$ be an irreducible (1,1)-divisor on $X$. Then $(X, \frac{1}{p}D)$ is a Fano weighted surface. Let $A$ be the corresponding standard GL-order and $\mathbb{L} = \Z (1,0) + \Z(\frac{1}{p},\frac{1}{p})$. In this case, there is no harm in writing suggestively the functor $(-)(\frac{1}{p}D) = (-)(\frac{1}{p},\frac{1}{p})$. Then $\mathbb{L}$ identifies with the divisor class group of $(X,\tfrac{1}{p}D)$ introduced in Section~\ref{scox} and the set of isomorphism classes of line bundles is $\{\calo_A(a,b) | (a,b) \in \mathbb{L}\}$.  Furthermore, Proposition~3.2ii) reduces to the formula
\begin{equation} \label{ecomputefor11ram}
\shom_A(\calo_A, \calo_A(a,b)) = \ox(\lfloor a \rfloor, \lfloor b \rfloor)
\end{equation}

\end{eg}

\begin{prop}  \label{p11ram}  \mlabel{p11ram}
With the above notation, 
$$\calt = \bigoplus_{i=0}^p \calo_A(i/p,i/p)  \oplus \bigoplus_{i=0}^p \calo_A(i/p,1+i/p)$$
is a 2-hereditary tilting bundle on $(X, \frac{1}{p}D)$. 
\end{prop}
\begin{proof} We first observe that $\Ext^i_A(\calt, \omega_A^{-r} \otimes_A \calt) = 0$ for $i=1,2, r\geq 0$. Indeed, $\omega_A^{-1} = A(1+1/p,1+1/p)$ so Formula~(\ref{ecomputefor11ram}) and Proposition~\ref{plinebundles} readily give the desired Ext vanishing. 

To establish the generation condition, we let ${\sf C}$ denote the category generated by $\calt$ and show ${\sf C}$ contains the generators of Theorem~\ref{tgen} for appropriate choice of $\mathcal{T}_I$. Using the notation of that theorem, we consider first $I = \varnothing$ and note that $\calt_I = \ox \oplus \ox(0,1) \oplus \ox(1,1) \oplus \ox(1,2)$ works for $\calo_A \otimes_X \calt_I$ is a direct summand of $\calt$. For the other case $I = \{1\}$, we note $D \simeq \PP^1$ and let $\calt_I = \calo_D(2) \oplus \calo_D(3)$. For $j = 0,1,\ldots, p-2, \ \eps = 0,1$, we consider a non-zero map 
$$\iota_{j,\eps}: \calo_A(1 - \tfrac{j+1}{p}, 1 + \eps - \tfrac{j+1}{p})  \lm \calo_A(1 - \tfrac{j}{p},1 + \eps - \tfrac{j}{p})  .$$
Then ${\sf C}$ contains $\coker \iota_{j,\eps}\simeq \calo_{D_I}(-\frac{j}{p}D) \otimes_X \calo_D(2 + \eps)$, which are the other generators listed in Theorem~\ref{tgen}. 
\end{proof} 
\vs

It is easy to compute the endomorphism ring of the 2-hereditary tilting bundle $\calt$ in Proposition~\ref{p11ram}. Let $u,v \in \Hom_X(\ox(0,1),\ox(1,1)), x,y \in \Hom_X(\ox,\ox(0,1))$ be bases and suppose the weighted divisor $D$ is defined by $F(u,v,x,y) = 0$. The basis $x,y$ gives compatible bases $x_i,y_i \in \Hom_A(\calo_A(\tfrac{i}{p},\tfrac{i}{p}), \calo_A(\tfrac{i}{p},1+\tfrac{i}{p}))$. We may similarly pick coherent bases $t_{ij} \in \Hom_A(\calo_A(\tfrac{i}{p},\tfrac{j}{p}), \calo_A(\tfrac{i+1}{p},\tfrac{j+1}{p}))$. These give the generators for $\End_A \calt$. We illustrate for the case $p=2$ with, as is customary, subscripts on arrows dropped

$$\diagram
\calo_A(0,1) \drrto<.5ex>^(.35)u\drrto<-.5ex>_(.35)v \rto^t & \calo_A(\tfrac{1}{2},\tfrac{3}{2}) \rto^t & \calo_A(1,2) \\
\calo_A \rto_(.4)t \uto<.5ex>^(.35)x\uto<-.5ex>_(.35)y & \calo_A(\tfrac{1}{2},\tfrac{1}{2}) \rto_t \uto<.5ex>^(.3)x\uto<-.5ex>_(.3)y & 
\calo_A(1,1) \uto<.5ex>^(.35)x\uto<-.5ex>_(.35)y 
\enddiagram$$
For general $p$, the relations are 
\begin{equation}
 xt = tx, \quad yt = ty, \quad xuy = yux, \quad xvy = yvx, \quad t^p = F(u,v,x,y).
\end{equation}
\vs

We turn now our attention to weighting the blowup $X$ of $\PP^2$ at a single point. Let $E \subset X$ be the exceptional curve and $H \subset X$ be the pullback of a generic line in $\PP^2$.  Note that $X \simeq \PP_{\PP^1}(\calo \oplus \calo(-1))$ and from this viewpoint we have $\ox(E) = \calo_{X/\PP^1}(1)$ and the generic fibre $F$ of $\pi: X \lm \PP^1$ is linearly equivalent to $H-E$. We need the following standard cohomology computations. 

\begin{lemma}  \label{lcohomologyF1}  \mlabel{lcohomologyF1}
$0 = H^1(\ox(aH + bE)) = H^2(aH + bE)$ if any of the following hold:
\begin{enumerate}
 \item $a+b >-1$ and $b \leq 1$,
 \item $a+b=-1$, or
 \item $a=-2, b=0$.
\end{enumerate}
\end{lemma}
\begin{proof}
We let $e = a+b$ so $eE + aF \sim aH + bE$. For $e \geq -1$ we have $R^1\pi_*(\ox(eE + aF)) = 0$ by \cite[Lemma~V.2.4]{Har} (the proof also works when $e=-1$). Hence $H^2(\ox(eE+aF)) = 0$. Also, by the Leray-Serre spectral sequence, projection formula and \cite[Proposition~II.7.11]{Har} we find
$$ H^1(\ox(eE+aF)) = H^1(\pi_*(\ox(eE+aF))) = H^1(S^e(\calo_{\PP^1} \oplus\calo_{\PP^1}(-1)) \otimes_{\PP^1} \calo_{\PP^1}(a)) $$
where $S^e$ denotes the $e$-th symmetric power. Now $S^e=0$ if $e=-1$ so case~ii) follows. For $e>-1$, we may expand $S^e$ and see that cohomology vanishes when $-b=a-e \geq -1$ and case~i) follows. Finally, we can check $0 = H^1(\ox(-2H)) = H^2(\ox(-2H))$ using the Leray-Serre spectral sequence for the blow down. Alternatively, we note that in this case $\pi_*(\ox(-2H)) = 0,\ R^1\pi_*(\ox(-2H)) \simeq \calo_{\PP^1}(-1)$. 
\end{proof}
\vs

\begin{eg}
We let $X$ be the blowup of $\PP^2$ at a single point and weight $X$ at the strict transform $H$ of a line which does not pass through this point. Let $p$ be the weight and $A$ be the corresponding standard GL-order. 
\end{eg}
\begin{prop}  \label{pF1}  \mlabel{pF1}
With the above notation, 
$$\calt = \calo_A(E) \oplus \calo_A \oplus \calo_A(\tfrac{1}{p}H) \oplus \ldots \oplus \calo_A(2H) $$ 
is a 2-hereditary tilting bundle on $X$. 
\end{prop}
\begin{proof}
Arguing as in the proof of Proposition~\ref{p11ram} using Theorem~\ref{tgen}, we see that $\calt$ generates the derived category. Let $\calo_A(D), \calo_A(D')$ be two line bundle summands of $\calt$. Going through the various cases we see that $\shom_A(\calo_A(D),\calo_A(D')) \simeq \ox(a'H + b'E)$ where $a'+b' \geq -2$ with equality only when $(a',b') = (-2,0)$. The partial tilting condition follows now from Lemma~\ref{lcohomologyF1}. Suppose now $r>0$ and let 
$$\shom_A(\calo_A(D),\omega_A^{-r} \otimes_A \calo_A(D')) \simeq \ox(aH + bE).$$
Now $\omega_A^{-1} \otimes_A (-)$ is the shift by $(2+\tfrac{1}{p})H -E$ so 
$$a \geq a' + 2r + \lfloor \tfrac{r}{p} \rfloor, \quad b = b' - r .$$
It follows that $a+b > a'+b', \ b < b'$ so Lemma~\ref{lcohomologyF1} gives the desired cohomology vanishing. 
\end{proof}
\vs

Again, $\End_A \calt$ is easy to compute. There are generators i) $x \in \Hom_A(\calo_A(\tfrac{i}{p}H), \calo_A(\tfrac{i+1}{p}H))$ for $i = 0,1,\ldots, 2p-1$, ii)  $y,z \in \Hom_A(\calo_A(\tfrac{j}{p}H), \calo_A((\tfrac{j}{p}+1)H))$ for $j=1, \ldots, p$, iii) $u \in \Hom_A(\calo_A, \calo_A(E))$ and iv) $y',z' \in \Hom_A(\calo_A(E), \calo_A(H))$. 
We draw the quiver in the case $p=2$

$$\diagram
 & \calo_A(\tfrac{1}{2}H) \drto^x \rrto<.5ex>^y \rrto<-.5ex>_z& & \calo_A(\tfrac{3}{2}H) \drto^x & \\
\calo_A \rto^u \urto^x & \calo_A(E) \rto<.5ex>^{y'} \rto<-.5ex>_{z'} & \calo_A(H) \urto^x \rrto<.5ex>^y \rrto<-.5ex>_z& & \calo_A(2H)
\enddiagram$$

The relations are
\begin{equation}
x y = y x, \quad xz = zx, \quad y'z = z'y, \quad uy'x = xy, \quad uz'x = xz.
\end{equation}

\vs
\begin{eg}
We now consider the case where $X = \PP^2$ is weighted on 4 lines $H_1,H_2,H_3,H_4$ in general position, each with weight 2. Let $A$ be the corresponding GL-order and let $H$ be a general line in $X$. As in Section~\ref{scox}, let $\mathbb{L}$ be the abelian group generated by $\tfrac{1}{2}H_1, \tfrac{1}{2}H_2, \tfrac{1}{2}H_3, \tfrac{1}{2}H_4$ modulo the linear equivalence relation $H_1 \sim H_2 \sim H_3 \sim H_4$. There is an additive {\em degree} function $\deg\colon \mathbb{L} \lm \tfrac{1}{2} \Z$ defined by $\deg \tfrac{1}{2}H_i = \tfrac{1}{2}$. The group $\mathbb{L}$ is partially ordered by $D \leq D'$ if $\Hom_A(\calo_A(D),\calo_A(D')) \neq 0$. We write $[0,2H]$ for set of $D \in \mathbb{L}$ with $0 \leq D \leq 2H$ and $(0,2H) = [0,2H] - \{0,2H\}$. We know from \cite[Theorem~6.1]{HIMO} and \cite[Theorem~2.2]{IL} that $\bigoplus_{D \in [0,2H]} \calo_A(D)$ is a tilting bundle on $A$, although not a 2-hereditary one. We use mutations to alter this into a 2-hereditary tilting bundle. To this end, consider the following short exact sequences which define the rank 3 bundles $\Omega$ and $\Xi$ below.
\begin{gather}   \label{eomega}
E_{\Omega}: 0 \lm \Omega \lm \bigoplus_{i=1}^4 \calo_A(H + \tfrac{1}{2}H_i) \xrightarrow{\pi} \calo_A(2H) \lm 0 \\  \label{exi}
E_{\Xi}: 0  \lm \calo_A \xrightarrow{\iota} \bigoplus_{i=1}^4 \calo_A(\tfrac{1}{2}H_i) \lm \Xi \lm 0
\end{gather}
where $\pi, \iota$ are induced by the natural inclusion of line bundles. 
\end{eg}

\begin{thm}  \label{t2222}  \mlabel{t2222}
The bundle 
$$\calt = \Xi \oplus \bigoplus_{D \in (0,2H)} \!\!\calo_A(D)\ \oplus \Omega$$
is a 2-hereditary tilting bundle on $A$. 
\end{thm}
\begin{proof}
First note that the triangulated category generated by $\calt$ contains the tilting bundle $\bigoplus_{D \in [0,2H]} \calo_A(D)$ and hence $D^b_c(A)$. It remains only to check the vanishing of appropriate Ext groups. This will be achieved in a series of lemmas.
\begin{lemma}  \label{lhompi}  \mlabel{lhompi}
Let $D \in \mathbb{L}$. Then $\phi = \Hom_A(\calo_A(D), \pi)$ is surjective unless $D = 2H$, in which case $\coker \phi = k$. 
\end{lemma}
\begin{proof}
Consider the inclusion map $\pi_i \colon \calo_A(H + \tfrac{1}{2}H_i) \hookrightarrow  \calo(2H)$. If $\Hom_A(\calo_A(D), \calo_A(2H))$ is the line bundle $\call$ on $X$, then $\Hom_A(\calo_A(D), \pi_i)$ is either the inclusion $\call(-H_i) \hookrightarrow \call$, or the identity map $\call \lm \call$. Furthermore, we know from the Koszul resolution of the polynomial ring that $H^0(\oplus \call(-H_i)) \lm H^0(\call)$ is surjective unless $\call \simeq \ox$. In the latter case, we must have either a) $D = 2H$  in which case we are done, or b) $D \leq 2H-\tfrac{1}{2}H_i= H + \tfrac{1}{2}H_i$ for some $i$. In this case, we see $\Hom_A(\calo_A(D), \pi_i)$ is already an isomorphism.  
\end{proof}
\vs

Note that $-K_A = -H + \tfrac{1}{2}\sum H_i$ has degree 1 and that $-2K_A \sim 2H$. 
The next lemma follow froms Proposition~\ref{plinebundles} and Serre duality.
\begin{lemma}  \label{lextOD}  \mlabel{lextOD}
Let $D,D' \in [0,2H]$ and $r\geq 0$. Then $\Ext^1_A(\calo_A(D), \calo_A(D'-rK_A)) = 0$ and 
$$
\Ext^2_A(\calo_A(D), \calo_A(D'-rK_A)) = 
\begin{cases} 
k & \text{ if } r = 1, D = 2H \text{ and } D' = 0 \\
0 & \text{ else} 
\end{cases}
$$
\end{lemma}
\begin{lemma}  \label{lextODomega}  \mlabel{lextODomega}
Let $D \in [0,2H]$ and $r \geq 0$. Then 
\begin{enumerate}
 \item $\Ext^2(\Omega,\calo_A(D - rK_A)) = 0 = \Ext^2(\calo_A(D), \Omega(- rK_A))$.
 \item $\Ext^1(\Omega,\calo_A(D - rK_A)) = 
\begin{cases}
 k & \text{ if } r=1 \text{ and } D=0 \\
 0 & \text{ else}
\end{cases}
$.
\item $\Ext^1(\calo_A(D), \Omega(- rK_A)) = 
\begin{cases}
 k & \text{ if } r=0 \text{ and } D=2H \\
 0 & \text{ else}
\end{cases}
$.
\item $\Ext^2(\Omega,\Omega(- rK_A)) = \Ext^1(\Omega, \Omega(- rK_A)) =0$.
\end{enumerate}
 
\end{lemma}
\begin{proof}
For the most part, these follow from the long exact sequence associated to twists of $E_{\Omega}$ (see \eqref{eomega}) and Lemma~\ref{lextOD}. We prove only part~iii) which requires further attention. Now
$$ \Ext^1(\calo_A(D), \Omega(- rK_A)) \simeq \Ext^1(\calo_A(D+rK_A), \Omega) \simeq \coker \Hom_A(\calo_A(D+rK_A),\pi).$$
By Lemma~\ref{lhompi}, we know this is zero unless $D + rK_A = 2H$. In this case, $D \in [0,2H] \cap (2H - \mathbb{N} K_A)$. Now $K_A$ has degree -1, so this can only occur when $r=0$. 
\end{proof}
\vs

We omit the proof of the following ``dual'' result involving $\Xi$.
\begin{lemma}  \label{lextODxi}  \mlabel{lextODxi}
Let $D \in [0,2H]$ and $r \geq 0$. Then 
\begin{enumerate}
 \item $\Ext^2(\Xi,\calo_A(D - rK_A)) = 0 = \Ext^2(\calo_A(D), \Xi(- rK_A))$.
 \item $\Ext^1(\Xi,\calo_A(D - rK_A)) = 
\begin{cases}
 k & \text{ if } r=0 \text{ and } D=0 \\
 0 & \text{ else}
\end{cases}
$.
\item $\Ext^1(\calo_A(D), \Xi(- rK_A)) = 
\begin{cases}
 k & \text{ if } r=1 \text{ and } D=2H \\
 0 & \text{ else}
\end{cases}
$.
\item $\Ext^2(\Xi,\Xi(- rK_A)) = \Ext^1(\Xi, \Xi(- rK_A)) =0$.
\end{enumerate} 
\end{lemma}
The long exact sequences associated to $E_{\Omega}$ and $E_{\Xi}$ (see \eqref{eomega},\eqref{exi}) and Lemmas~\ref{lextODomega}, \ref{lextODxi} show that 
$$\Ext^2(\Omega,\Xi(- rK_A))  = \Ext^1(\Omega,\Xi(- rK_A))= \Ext^2(\Xi, \Omega(- rK_A)) =0$$
whilst
$$ \Ext^1(\Xi, \Omega(- rK_A)) = \coker \Hom_A(\iota, \Omega(-rK_A)).$$
The theorem will thus be proved once we show 
\begin{lemma} \label{lextxiomega}  \mlabel{lextxiomega}
 The cokernel of $\Hom_A(\iota, \Omega(-rK_A))$ is 0. 
\end{lemma}
\begin{proof}
We work in the Cox ring $R= R_{\mathbb{X}}$ as defined in Section~\ref{scox}. This is generated by $x_i \in \Hom_A(\calo_A,\calo_A(\tfrac{1}{2}H_i)), \ i = 1,2,3,4$. Now we may naturally identify $x_i^2$ with a global section of $\ox(D_i)$ and, changing coordinates appropriately, we may assume that we have the relation $x_1^2 + x_2^2 + x_3^2 + x_4^2 = 0$. The Cox ring $R_X$ of $X$ is the subalgebra of $R$ corresponding to the {\em $H$-Veronese}, that is, the sum of all components of degree $nH, n \in  \mathbb{N}$. 

Recall that by default, $A$-modules are left modules so homomorphisms between line bundles are given by right multiplication by elements of $R$. We may thus view $\iota$ and $\pi$ as right multiplication by $\mathbf{x} = (x_1 \ x_2 \ x_3 \ x_4)$ and $\mathbf{x}^T$ respectively. 
Consider now a morphism $\calo_A \lm \bigoplus_i \calo_A(H + \tfrac{1}{2}H_i -rK_A)$ which we view as a 4-vector $\mathbf{f} = (f_1 \ f_2 \ f_3 \ f_4)$ with entries in $R$. It factors through a morphism $f \colon \calo_A \lm \Omega(-rK_A)$ if and only if $\mathbf{f}\mathbf{x}^T = 0$. 

In this case, we can factor $f$ through $\iota$ if there exist $\mathbf{f}_1,\mathbf{f}_3,\mathbf{f}_3,\mathbf{f}_4 \in R^4$ with $\mathbf{f}_i \mathbf{x}^T = 0$ for all $i$ and 
$$ \mathbf{f} =  x_1 \mathbf{f}_1 +  x_2 \mathbf{f}_2 + x_3 \mathbf{f}_3 + x_4 \mathbf{f}_4.$$
Indeed, $f = \phi \circ \iota$ where $\phi$ is induced by the $4 \times 4$-matrix 
$$
\begin{pmatrix}
\mathbf{f}_1 \\ \mathbf{f}_2 \\ \mathbf{f}_3  \\ \mathbf{f}_4
\end{pmatrix}.
$$
In this case, we shall say that the 4-vector $\mathbf{f}$ is {\em liftable}. 

Consider first the case where $r$ is even so $-rK_A \sim rH$. Then $\mathbf{f} = (f'_1  x_1  \ f'_2  x_2  \ f'_3  x_3  \ f'_4  x_4)$ for some $f'_i \in \ox(r+1) \subset R_X$. Computing in the polynomial ring $R_X$, we find that $\mathbf{f} \mathbf{x}^T= 0$ if and only if 
\begin{equation*}
\mathbf{f}' := (f'_1  \ f'_2  \ f'_3  \ f'_4) \in 
R_X(1 \ 1 \ 1 \ 1) \oplus R_X (0 \ x_3^2 \ -\!x_2^2 \ 0 )  \oplus R_X (0 \ 0 \ x_4^2 \ -\!x_3^2)  \oplus R_X (0 \ x_4^2 \ 0 \ -\!x_2^2) 
\end{equation*}
Hence
$$ \mathbf{f} \in R_X(x_1 \ x_2 \ x_3 \ x_4) \oplus R_X (0 \ x_2x_3^2 \ -\!x_3x_2^2 \ 0 )  \oplus R_X (0 \ 0 \ x_3x_4^2 \ -\!x_4x_3^2)  \oplus R_X (0 \ x_2x_4^2 \ 0 \ -\!x_4x_2^2) 
$$
Now $(0 \ x_2x_3^2 \ -\!x_3x_2^2 \ 0 ), (0 \ 0 \ x_3x_4^2 \ -\!x_4x_3^2), (0 \ x_2x_4^2 \ 0 \ -\!x_4x_2^2)$ are all liftable, so any $R_X$-linear combination of them is too. If $R_{X,>0}$ denotes the augmentation ideal consisting of positive degree elements, then all vectors in $R_{X,>0}(x_1 \ x_2 \ x_3 \ x_4)$ are also liftable, so as $r\geq 0$, the theorem is proved in the $r$ even case.

Consider now the case where $r$ is odd so $-rK_A \simeq (r-2)H + \tfrac{1}{2}\sum H_i$. Then 
$$\mathbf{f} = (f'_1  x_2x_3x_4  \ f'_2  x_1x_3x_4  \ f'_3  x_1x_2x_4  \ f'_4  x_1x_2x_3)$$ 
for some $f'_i \in \ox(r) \subset R_X$. This time $\mathbf{f} \mathbf{x}^T= 0$ amounts to $\sum f'_i = 0$. Hence 
$$ \mathbf{f}' := (f'_1  \ f'_2  \ f'_3  \ f'_4) \in 
R_X(1 \ -1 \ 0 \ 0) \oplus R_X (0 \ 1 \ -1 \ 0 )  \oplus R_X (0 \ 0 \ 1 \ -1)
$$
and 
$$\mathbf{f} \in R_X(x_2x_3x_4 \ -\!x_1x_3x_4 \ 0 \ 0) \oplus R_X (0 \ x_1x_3x_4 \ -\!x_1x_2x_4 \ 0 )  \oplus R_X (0 \ 0 \ x_1x_2x_4 \ -\!x_1x_2x_3)
$$
Since each of the $R_X$-generators on the right is liftable, we are done in this case too.
\end{proof}
\vs

The proof of the theorem is now complete. 
\end{proof}

\end{document}